\documentclass[a4paper]{amsart}

\makeatletter
\renewcommand\part{%
   \if@noskipsec \leavevmode \fi
   \par
   \addvspace{4ex}%
   \@afterindentfalse
   \secdef\@part\@spart}

\def\@part[#1]#2{%
    \ifnum \c@secnumdepth >\m@ne
      \refstepcounter{part}%
      \addcontentsline{toc}{part}{\thepart\hspace{1em}#1}%
    \else
      \addcontentsline{toc}{part}{#1}%
    \fi
    {\parindent \z@ \raggedright
     \interlinepenalty \@M
     \normalfont
\center{
     \ifnum \c@secnumdepth >\m@ne
       \large\bfseries \partname\nobreakspace\thepart: 
     \fi
     \large \bfseries #2}%
     \par}%
    \nobreak
    \vskip 3ex
    \@afterheading}
\def\@spart#1{%
    {\parindent \z@ \raggedright
     \interlinepenalty \@M
     \normalfont
     \huge \bfseries #1\par}%
     \nobreak
     \vskip 3ex
     \@afterheading}
\makeatother

\usepackage{microtype}

\usepackage{wrapfig,blindtext}

\usepackage{amsmath,amstext,amssymb,mathrsfs,amscd,amsthm,indentfirst, bbm}
\usepackage{amsfonts}

\usepackage{booktabs}
\usepackage{verbatim}
\usepackage{enumerate}
\usepackage{graphicx}
\usepackage{textcomp}
\usepackage{mathtools}

\usepackage{tikz, tikz-cd}
\usetikzlibrary{matrix,calc,arrows}

\newcommand{\bF}{\mathbb{F}}

\newcommand{\bN}{\mathbb{N}}

\newcommand{\bQ}{\mathbb{Q}}

\newcommand{\bZ}{\mathbb{Z}}

\newcommand{\bk}{\mathbbm{k}}

\newcommand\lra{\longrightarrow}

\newcommand\colim{\operatorname*{colim}}
\newcommand\hocolim{\operatorname*{hocolim}}
\newcommand\holim{\operatorname*{holim}}
\newcommand\Coker{\operatorname*{Coker}}
\newcommand\Ker{\operatorname*{Ker}}

\newcommand{\Hom}{\mathrm{Hom}}
\newcommand{\Ext}{\mathrm{Ext}}

\renewcommand{\epsilon}{\varepsilon}

\newcommand{\GL}{\mathrm{GL}}
\newcommand{\SL}{\mathrm{SL}}

\newcommand{\ul}{\underline}

\newcommand{\K}{\mathrm{K}}
\newcommand{\SK}{\mathrm{SK}}

\newcommand{\et}{\text{{\'e}t}}

\mathchardef\ordinarycolon\mathcode`\:
\mathcode`\:=\string"8000
\begingroup \catcode`\:=\active
  \gdef:{\mathrel{\mathop\ordinarycolon}}
\endgroup

\usepackage[all]{xy}
\CompileMatrices

\allowdisplaybreaks

\usepackage{amsthm}
\theoremstyle{plain}
\newtheorem{MainThm}{Theorem}
\newtheorem{MainCor}[MainThm]{Corollary}

\newtheorem{theorem}{Theorem}[section]

\newtheorem{proposition}[theorem]{Proposition}
\newtheorem{lemma}[theorem]{Lemma}

\newtheorem{corollary}[theorem]{Corollary}

\theoremstyle{definition}

\theoremstyle{remark}

\newtheorem{remark}[theorem]{Remark}
\newtheorem*{remark*}{Remark}
\newtheorem{example}[theorem]{Example}
\newtheorem*{example*}{Example}
\newtheorem{question}{Question}

\numberwithin{equation}{section}

\usepackage[hypertexnames=false]{hyperref}

\title{Stable cohomology of congruence subgroups}

\author{Oscar Randal-Williams}
\email{o.randal-williams@dpmms.cam.ac.uk}
\address{Centre for Mathematical Sciences\\
Wilberforce Road\\
Cambridge CB3 0WB\\
UK}

\begin{document}
\begin{abstract}
We describe the $\bF_p$-cohomology of the congruence subgroups $\SL_n(\bZ, p^m)$ in degrees $* < p-1$, for all large enough $n$, establishing a formula proposed by F.~Calegari. Along the way, we also establish a formula for the stable cohomology of $\SL_n(\bZ/p)$ with certain twisted coefficients.
\end{abstract}
\maketitle

\setcounter{tocdepth}{1}
\tableofcontents

\section*{Introduction}

\subsection*{Stable cohomology of congruence subgroups}
Let $p$ be an odd prime number. We will be concerned with the $\bF_p$-cohomology of the level $p^m$ congruence subgroups
$$\SL_n(\bZ, p^m) := \Ker(\SL_n(\bZ) \to \SL_n(\bZ/p^m))$$
with $m \geq 1$, in a range of cohomological degrees which is stable in two senses: certainly $n$ should be large compared with the cohomological degree, but $p$ should be too. Our main result is expressed in terms of the completed cohomology 
$$\widetilde{H}^*(\SL_n;\bF_p) := \colim_m H^*(\SL_n(\bZ, p^m);\bF_p)$$ 
of Calegari and Emerton \cite{CESurvey}, and should be considered as extending some of the ideas developed by Calegari \cite{Calegari} to higher cohomological degrees.

\begin{MainThm}\label{thm:mainCong}
Let $p$ be odd. In degrees $* < p-1$ and for all large enough $n$ there is an isomorphism
$$H^*(\SL_n(\bZ, p^m);\bF_p) \cong \Lambda_{\bF_p}^*[sl_n(\bF_p)^\vee] \otimes \widetilde{H}^*(\SL_n;\bF_p)$$
of $\bF_p$-algebras and of $\SL_n(\bZ/p^m)$-representations.
\end{MainThm}

In Remark \ref{rem:StabRangeThmA} we explain that $n \geq 2p+4$ suffices.

Here $sl_n(\bF_p)$ denotes the vector space of traceless $n \times n$ matrices, considered as an $\SL_n(\bZ/p^m)$-representation via $\SL_n(\bZ/p^m) \to \SL_n(\bZ/p)$ and the adjoint action. That the dimensions of the cohomology groups $H^i(\SL_n(\bZ, p^m);\bF_p)$ are eventually polynomial in $n$, and admit a systematic description, is the fact that they satisfy representation stability \cite[Theorems 1.5 and 1.6]{CEFN}. The point of Theorem \ref{thm:mainCong} is that it identifies what the ``representation stable'' cohomology is, at least in degrees $* < p-1$. Earlier results we are aware of that calculate the (co)homology of $\SL_n(\bZ, p)$ in a stable range of degrees are those of Lee--Szczarba \cite[Theorem 1.1]{LeeSzczarba} in degree $*=1$, and of Calegari \cite[Corollary 4.4]{Calegari} in degrees $* \leq 2$, whose methods we extend here.

Crucial to the proof of this theorem, as well as to its applications, is that completed cohomology satisfies homological stability with respect to $n$ \cite{CEStability}, and the relation between the stable completed cohomology and the fibre of the $p$-adic completion map in algebraic $K$-theory \cite{Calegari}. More precisely, if
$$\kappa : \SK(\bZ ; \bZ_p) \lra \SK(\bZ_p ; \bZ_p)$$
is the map induced by $p$-adic completion $\bZ \to \bZ_p$ on (the 1-connected cover of) $p$-adic $K$-theory, then Calegari proves that
$$\widetilde{H}^*(\SL;\bF_p) \cong H^*(\Omega^\infty \mathrm{hofib}(\kappa) ; \bF_p).$$
Combined with deep results in algebraic $K$-theory, this can be used to evaluate completed cohomology for $p$ a regular prime, leading to the following formula.

\begin{MainCor}\label{cor:mainReg}
Let $p$ be an odd regular prime. Then in degrees $* < p-1$ and for all large enough $n$ there is an isomorphism
$$H^*(\SL_n(\bZ, p^m);\bF_p) \cong \Lambda_{\bF_p}^*[sl_n(\bF_p)^\vee] \otimes \bF_p[x_2, x_6, x_{10}, \ldots]$$
of $\bF_p$-algebras and of $\SL_n(\bZ/p^m)$-representations.
\end{MainCor}

Recall that an odd prime is called \emph{regular} if it does not divide the numerator of a Bernoulli number. We will explain in Sections \ref{sec:CompletionMap} and \ref{sec:LFunction} that regularity simplifies matters in two ways: firstly $K_{4i+2}(\bZ;\bZ_p)=0$ if $p$ is regular, and secondly the map $\kappa_*: \bZ_p \cong K_{4i+1}(\bZ ; \bZ_p) \to K_{4i+1}(\bZ_p ; \bZ_p) \cong \bZ_p$ is an isomorphism if $p$ is regular. Together these facts cause $H^*(\Omega^\infty \mathrm{hofib}(\kappa) ; \bF_p)$ to have the simple description appearing in Corollary \ref{cor:mainReg}. The situation can also be analysed at irregular primes: we will justify the following examples in Section \ref{sec:LFunction}.

\begin{example*}
For $p=37$ we have
$$H^*(\SL_n(\bZ, p^m);\bF_p) \cong \Lambda_{\bF_p}^*[sl_n(\bF_p)^\vee] \otimes \bF_p[x_2, x_6, x_{10}, \ldots] \otimes \bF_p[y_8] \otimes \Lambda^*_{\bF_p}[y_9]$$
in degrees $* < 36$ for all large enough $n$.
\end{example*}

\begin{example*}
For $p=16843$ we have
$$H^*(\SL_n(\bZ, p^m);\bF_p) \cong \Lambda_{\bF_p}^*[sl_n(\bF_p)^\vee] \otimes \bF_p[x_2, x_6, x_{10}, \ldots] \otimes \bF_p[y_4] \otimes \Lambda^*_{\bF_p}[y_5]$$
in degrees $* < 16842$ for all large enough $n$.
\end{example*}

\begin{example*}
For $p=2124679$ we have
\begin{align*}
H^*(\SL_n(\bZ, p^m);\bF_p) \cong \Lambda_{\bF_p}^*[sl_n(\bF_p)^\vee] &\otimes \bF_p[x_2, x_6, x_{10}, \ldots] \otimes \bF_p[y_4] \otimes \Lambda^*_{\bF_p}[y_5] 
\end{align*}
in degrees $* < 1403794$ for all large enough $n$.
\end{example*}

For completeness, in Section \ref{sec:AwayFromChar} we describe the analogue of Theorem \ref{thm:mainCong} with coefficients coprime to $p$, whose statement is simpler and whose proof is much simpler.

\subsection*{Strategy}
The general strategy for proving Theorem \ref{thm:mainCong} is the same as \cite[Section 3]{Calegari}. The theory of completed cohomology provides a spectral sequence
$$E_2^{s,t} = H^s_\mathrm{cts}(\SL_n(\bZ_p, p^m) ; \bF_p) \otimes \widetilde{H}^t(\SL_n;\bF_p) \Longrightarrow H^{s+t}(\SL_n(\bZ, p^m) ; \bF_p),$$
and the theory of $p$-adic analytic groups gives an identification
$$H^*_\mathrm{cts}(\SL_n(\bZ_p, p^m) ; \bF_p) \cong \Lambda_{\bF_p}^*[sl_n(\bF_p)^\vee]$$
of $\bF_p$-algebras and of $\SL_n(\bZ/p^m)$-representations. One must then show that this spectral sequence collapses at $E_2$ in degrees $* < p-1$, and that it has no nontrivial extensions either multiplicatively or as $\SL_n(\bZ/p^m)$-representations. This is what we shall do.

That a statement like Corollary \ref{cor:mainReg} could be true we learnt from a talk given by Calegari at BIRS in October 2021 \cite{CalegariTalk}. Based on heuristics including that Corollary \ref{cor:mainReg} should be true, Calegari presented a conjectural formula for the cohomology of the finite groups $\SL_n(\bZ/p)$ with coefficients in certain modular representations (coming from representations of the algebraic group $\SL_n$), and suggested that such  formula could be useful in approaching results like Corollary \ref{cor:mainReg}. The second thing we do in this paper is to prove this conjectural formula (in Theorem \ref{thm:mainFin} below). We will not directly use this formula to prove Theorem \ref{thm:mainCong}, but we will use many of the same ingredients that go into proving it. We formulate it in the following section.

\subsection*{Stable twisted cohomology of $\SL_n(k)$}
In this section we work not just with $\bF_p$ but with a finite field $k$ of characteristic $p$. We work throughout with $k$-modules, and in particular form all tensor products over $k$. If $V$ is a finite-dimensional $k$-module with dual $V^\vee$ and coevaluation map $coev: k \to V \otimes V^\vee$, then we may form the quotient $V_{[n,m]}$ of $V^{\otimes n} \otimes (V^\vee)^{\otimes m}$ by the subspace spanned by inserting coevaluations in all possible ways. The group $\Sigma_n \times \Sigma_m$ acts on $V^{\otimes n} \otimes (V^\vee)^{\otimes m}$ by permuting the factors, and this action descends to $V_{[n,m]}$. For partitions $\lambda \vdash n$ and $\mu \vdash m$ with associated Specht modules $S^\lambda$ and $S^\mu$, we define
$$S_{\lambda, \mu}(V) := \Hom_{k[\Sigma_n \times \Sigma_m]} (S^\lambda \otimes S^\mu, V_{[n,m]}).$$
This is a $\GL(V)$-representation.

The formula proposed by Calegari is then as follows. Form the graded algebra $\mathrm{Sym}^\bullet(V \otimes V^\vee)$, where $V \otimes V^\vee$ is placed in degree 2, and let $X^\bullet$ denote its quotient by the ideal generated by the $\GL(V)$-invariant elements\footnote{We will see (in Lemma \ref{lem:invariants}) that the invariant elements form a polynomial algebra $k[c_1, c_2, c_3, \ldots]$ in a range of degrees increasing with $\dim(V)$, with $|c_i|=2i$.}.

\begin{MainThm}\label{thm:mainFin}
For all partitions $\lambda \vdash n$ and $\mu \vdash m$ with  $n+m \leq \tfrac{p+1}{2}$, there is an isomorphism
$$H^i(\SL(V) ; S_{\lambda, \mu}(V)) \cong [S_{\lambda, \mu}(V) \otimes X^i]^{\SL(V)}$$
in degrees $i < 2p$, as long as $\dim(V)$ is large enough.
\end{MainThm}

In Remark \ref{rem:StabRangeThmC} we explain that $\dim(V) \geq 2i+2+n+m$ suffices.

\begin{remark}
The analogous statement for $\GL(V)$ instead of $\SL(V)$ holds too, and is in fact what we shall focus on: the statement for $\SL(V)$ will follow because the $k^\times$-action on $H^*(\SL(V) ; V_{[n,m]})$ is trivial in a stable range.
\end{remark}

\subsection*{Questions} Theorem \ref{thm:mainCong} immediately suggests some avenues for further research. Some of the most obvious are:

\begin{question}
How does $H^*(\SL_n(\bZ, p^m) ; \bF_p)$ behave beyond the range $* < p-1$?
\end{question}

\begin{question}
If $\mathcal{O}_K$ is the ring of integers in a number field, and $\mathfrak{p} \subset \mathcal{O}_K$ is a prime ideal lying over $(p) \subset \mathbb{Z}$, is there a corresponding description of $H^*(\SL_n(\mathcal{O}_K, \mathfrak{p}^m) ; \bF_p)$ in a stable range?
\end{question}

One should read \cite[Section 4]{Calegari} to get started with this.

\begin{question}
What is the analogous result for other congruence subgroups of arithmetic groups, such as $\mathrm{Sp}_{2g}(\bZ, p^m)$?
\end{question}

\subsection*{Leitfaden} The reader interested only in the proof of Theorem \ref{thm:mainCong} may read Sections \ref{sec:Stability}--\ref{sec:wBr} and then skip to Section \ref{sec:Recollections}.

\subsection*{Acknowledgements} I thank the organisers of the BIRS workshop \emph{Cohomology of Arithmetic Groups: Duality, Stability, and Computations} (21w5011) for hosting such an intradisciplinary workshop, and F.\ Calegari for presenting the questions addressed here in his talk, providing detailed notes on it afterwards, and his helpful correspondence. I am grateful to C.\ Vespa for her guidance through the literature involved in Section \ref{sec:FunctorHomology}, and to the referees for their extremely useful input.

This material is based upon work supported by the Swedish Research Council under grant no.\ 2016-06596 while I was in residence at Institut Mittag-Leffler in Djursholm, Sweden during the semester \emph{Higher algebraic structures in algebra, topology, and geometry}. I was supported by the ERC under the European Union’s Horizon 2020 research and innovation programme (grant agreement No.\ 756444), and by a Philip Leverhulme Prize from the Leverhulme Trust.

\part{Stable twisted cohomology of $\SL_n(k)$}

In this part of the paper we work over a finite field $k$ of characteristic $p$.

\section{Stability considerations}\label{sec:Stability}

We will say that a cohomological statement about $\GL(V)$ or $\SL(V)$ holds ``in a stable range of degrees'' if it holds in a range of cohomological degrees tending to infinity with $\dim(V)$. For homological stability we will use the work of van der Kallen \cite{vanderkallen}.

\subsection{Coefficient systems}\label{sec:CoeffSys}

The modules $V$ and $V^\vee$ for $\GL(V)$ both fit into the setting of \cite[\S 5]{vanderkallen}. We recall this setting, in modernised language,  as follows. There is a category $\mathsf{VIC}(k)$ whose objects are the finite-dimensional $k$-modules, and whose morphisms from $V$ to $W$ are given by a linear injection $f : V \to W$ along with a choice of subspace $U \leq W$ complementary to $f(V)$. A \emph{coefficient system} will for us be defined  to be a functor
$$F : \mathsf{VIC}(k) \lra k\text{-Mod}.$$
Unravelling definitions, this provides a ``strongly central coefficient system'' in the sense of \cite[\S 5.2]{vanderkallen}. Here are two examples. The first is $I(V) := V$, with its functoriality given on a morphism $(f,U) : V \to W$ by $I(f,U) = f : V \to W$. The second is $I^\vee(V) := V^\vee$,  given on a morphism $(f,U) : V \to W$ by 
$$I^\vee(f,U) : V^\vee \overset{f^\vee}{\underset{\sim}\longleftarrow} f(V)^\vee \overset{inc}\lra f(V)^\vee \oplus U^\vee =  W^\vee.$$
From these, for finite sets $S$ and $T$ we can form coefficient systems $I^{\otimes S} \otimes (I^\vee)^{\otimes T}$, whose values at $V$ are of course $V^{\otimes S} \otimes (V^\vee)^{\otimes T}$. (Here the tensor product is, of course, taken over $k$.) Under the shifting operation $\Sigma$ on coefficient systems \cite[\S 5.3]{vanderkallen} we have a decomposition
\begin{equation}\label{eq:SplitCoeffSys}
\Sigma (I^{\otimes S} \otimes (I^\vee)^{\otimes T}) \cong \bigoplus_{\substack{A \subseteq S \\ B \subseteq T}} I^{\otimes A} \otimes (I^\vee)^{\otimes B},
\end{equation}
from which it quickly follows that $I^{\otimes S} \otimes (I^\vee)^{\otimes T}$ has degree $|S|+|T|$ in the sense of \cite[\S 5.5]{vanderkallen}. 

The evaluation maps yield commutative squares
\begin{equation*}
\begin{tikzcd}
V \otimes V^\vee \ar[rr, "ev_V"] && k \ar[dd, equals]\\
V \otimes f(V)^\vee \ar[u, swap, "id \otimes f^\vee"] \dar{f \otimes inc}\\
(f(V) \oplus U)\otimes(f(V)^\vee \oplus U^\vee) \ar[rr, "ev_{f(V) \oplus U}"] && k
\end{tikzcd}
\end{equation*}
giving a morphism of coefficient systems $ev : I \otimes I^\vee \to k$, the target being the constant coefficient system. We can therefore form the coefficient systems
$$I^{[S,T]} := \Ker\left(I^{\otimes S} \otimes (I^\vee)^{\otimes T} \to \bigoplus_{\substack{s \in S \\ t \in T}} I^{\otimes S-s} \otimes (I^\vee)^{\otimes T-t}\right).$$
The decomposition \eqref{eq:SplitCoeffSys} restricts to a decomposition $\Sigma(I^{[S,T]}) \cong \bigoplus_{{A \subseteq S, B \subseteq T}} I^{[A,B]}$, so $I^{[S,T]}$ also has degree $|S|+|T|$.

\begin{example}
The coefficient system $I^{[1,1]}$ associates to $V$ the vector space
$$sl(V) := \Ker(ev_V : V \otimes V^\vee \lra k).$$
When $V = k^n$ we will sometimes denote this $sl_n(k)$ instead.
\end{example}

We write $V^{[S,T]}$ for the value of the functor $I^{[S,T]}$ at $V$, given of course by the kernel $\Ker\left(V^{\otimes S} \otimes (V^\vee)^{\otimes T} \to \bigoplus_{\substack{s \in S, t \in T}} V^{\otimes S-s} \otimes (V^\vee)^{\otimes T-t}\right)$.

\subsection{Stability}\label{sec:stab}

By \cite[Theorem 5.6]{vanderkallen} (using that we can take $sdim=0$ as $k$ is a field \cite[\S 2.2]{vanderkallen}) it follows that for fixed finite sets $S$ and $T$ the maps
\begin{align*}
H_i(\SL(V) ; V^{\otimes S} \otimes (V^\vee)^{\otimes T}) &\lra H_i(\SL(V \oplus k) ; (V\oplus k)^{\otimes S} \otimes ((V\oplus k)^\vee)^{\otimes T})\\
H_i(\SL(V) ; V^{[S,T]}) &\lra H_i(\SL(V \oplus k) ; (V \oplus k)^{[S,T]})
\end{align*}
are isomorphisms as long as $2i \leq \dim(V)-2-(|S|+|T|)$. We will not keep track of explicit stability ranges, and will just use the fact that this range of degrees diverges with $\dim(V)$ for $S$ and $T$ fixed.

We can obtain a similar result on cohomology, by dualising. Recalling that in the introduction we defined
$$V_{[S,T]} = \Coker\left(\bigoplus_{\substack{s \in S \\ t \in T}} V^{\otimes S-s} \otimes (V^\vee)^{\otimes T-t} \to V^{\otimes S} \otimes (V^\vee)^{\otimes T}\right),$$
there is an isomorphism $V^{[T,S]} \overset{\sim}\to V_{[S,T]}^\vee$ induced by the (swapped) evaluation map $V^{\otimes S} \otimes (V^\vee)^{\otimes T} \overset{\sim}\to \mathrm{Hom}_k(V^{\otimes T} \otimes (V^\vee)^{\otimes S}, k)$. By the Universal Coefficient Theorem the maps
\begin{align*}
H^i(\SL(V \oplus k) ; (V\oplus k)^{\otimes T} \otimes ((V\oplus k)^\vee)^{\otimes S}) \lra & H^i(\SL(V) ; V^{\otimes T} \otimes (V^\vee)^{\otimes S}) \\
H^i(\SL(V \oplus k) ; (V \oplus k)_{[T,S]}) \lra & H^i(\SL(V) ; V_{[T,S]}) 
\end{align*}
are also isomorphisms in a stable range of degrees. Note that these maps are induced by the inclusions $\SL(V) \to \SL(V \oplus k)$ and the maps $(V\oplus k)^{\otimes T} \otimes ((V\oplus k)^\vee)^{\otimes S} \to V^{\otimes T} \otimes (V^\vee)^{\otimes S}$ \emph{dual} to the maps considered on homology (i.e.\ induced by projection $V \oplus k \to V$ on the first tensor factors and restriction $(V \oplus k)^\vee \to V^\vee$ on the second).

\subsection{SL vs.\ GL}\label{sec:SLvsGL}
The extension 
\begin{equation}\label{eq:ext}
1 \lra \SL(V) \lra \GL(V) \overset{\det}\lra k^\times \lra 1
\end{equation}
induces a $k^\times$-action on $H^*(\SL(V) ; V_{[S,T]})$.

\begin{lemma}\label{lem:ActTrivOnSL}
This action is trivial in a stable range of degrees.
\end{lemma}
\begin{proof}
Choose a decomposition $V=W \oplus k$. The cohomological stability result above applies to show that the natural map
$$H^*(\SL(W \oplus k) ; (W \oplus k)_{[S,T]}) \lra H^*(\SL(W) ; W_{[S,T]})$$
is an isomorphism in a stable range of degrees. For the $k^\times$-action on the source given by conjugation with $\mathrm{diag}(1_W, k^\times)$, and the trivial $k^\times$-action on the target, this map is equivariant, and so in the range in which this map is an isomorphism the $k^\times$-action on the domain is trivial.
\end{proof}

\begin{corollary}\label{cor:GLvSL}
The natural map
$$H^*(\GL(V) ; V_{[S,T]}) \lra H^*(\SL(V) ; V_{[S,T]})$$
is an isomorphism in a stable range of degrees.
\end{corollary}
\begin{proof}
As the finite field $k$ has characteristic $p$, its units $k^\times$ form a finite group whose order is invertible in $k$. It follows by transfer that $H^s(k^\times ; M)=0$ for $s>0$ and any $k[k^\times]$-module $M$. The spectral sequence for the extension \eqref{eq:ext} takes the form
$$E_2^{s,t} = H^s(k^\times ; H^t(\SL(V) ; V_{[S,T]})) \Longrightarrow H^{s+t}(\GL(V) ; V_{[S,T]})$$
and is therefore supported along the line $s=0$ so collapses to give an isomorphism $H^*(\GL(V) ; V_{[S,T]}) \overset{\sim}\to H^0(k^\times ; H^*(\SL(V) ; V_{[S,T]}))$. However, as the $k^\times$-action is trivial in a stable range of degrees, the latter is $H^*(\SL(V) ; V_{[S,T]})$ in this range.
\end{proof}

\section{Functor homology}\label{sec:FunctorHomology}

Our initial goal is to calculate $H^*(\GL(V) ; V^{\otimes S} \otimes (V^\vee)^{\otimes T})$, which we will do using methods of functor homology. We have attempted to keep the actual use of functor homology in the proofs, and to formulate statements only at the level of group (co)homology. There are no new ideas in the proofs, which simply combine results extracted from the functor homology literature. We are grateful to C.\ Vespa for explaining how to do so.

\subsection{$\Ext$ and products}
We will be interested in various cohomology groups of the form $H^*(\GL(V) ; U \otimes W^\vee)$ where $V$ is a finite-dimensional $k$-vector space and $U$ and $W$ are $k[\GL(V)]$-modules which are finite-dimensional as $k$-modules (typically $U$ and $W$ will be constructed functorially from $V$). It will be convenient to translate between such cohomology groups and $\Ext$-groups over $k[\GL(V)]$. 

For $k[\GL(V)]$-modules $A$ and $B$ we write $\Ext^i_{\GL(V)}(A,B) := \Ext^i_{k[\GL(V)]}(A,B)$, which we consider as the abelian group of morphisms in the derived category of $k[\GL(V)]$-modules from $A$ to the $i$-fold shift $B[i]$ (cf.\ e.g.\ \cite[\S 10.7]{Weibel}). The cocommutative coalgebra structure on the Hopf algebra $k[\GL(V)]$ defines a symmetric monoidal structure $- \otimes_k -$ on $k[\GL(V)]$-modules, which is exact in each variable and so descends to a symmetric monoidal structure on the derived category of $k[\GL(V)]$-modules (which preserves exact triangles in each variable).

Some of this structure can be spelled out as follows. If $A$, $B$, and $C$ are $k[\GL(V)]$-modules, then composition in the derived category yields the \emph{Yoneda product} \cite[\S 2.6]{BensonI}
$$-\circ- : \Ext_{\GL(V)}^i(B, C) \otimes \Ext_{\GL(V)}^j(A, B)  \lra \Ext_{\GL(V)}^{i+j}(A, C).$$
If $D$ is a further $k[\GL(V)]$-module then the symmetric monoidal structure $- \otimes_k -$ induces the \emph{cup product} \cite[\S 3.2]{BensonI}
$$-\otimes- : \Ext_{\GL(V)}^i(A, B) \otimes \Ext_{\GL(V)}^j(C, D) \lra \Ext_{\GL(V)}^{i+j}(A \otimes_k C, B \otimes_k D).$$
It is associative and graded commutative. Its compatibility with the Yoneda product is encoded by the fact that $- \otimes_k -$ defines a bifunctor on the derived category of $k[\GL(V)]$-modules.

As $H^*(\GL(V) ; U \otimes W^\vee) = \Ext^*_{\GL(V)}(k, U \otimes W^\vee)$ we can form the map
$$\Ext^*_{\GL(V)}(k, U \otimes W^\vee) \overset{- \otimes \mathrm{Id}_W}\lra \Ext^*_{\GL(V)}(W, U \otimes W^\vee \otimes W) \overset{(U \otimes ev) \circ -}\lra \Ext^*_{\GL(V)}(W, U)$$
using the evaluation $ev : W^\vee \otimes W \to k$, and this map is an isomorphism (with inverse given by composing $- \otimes \mathrm{Id}_{W^\vee}$ with $- \circ coev$ using $coev : k \to W \otimes W^\vee$).

\subsection{Defining cohomology classes}\label{sec:DefCohClasses}
As an instance of the discussion above, for a finite-dimensional $k$-vector space $V$ we have constructed an isomorphism
$$H^*(\GL(V) ; V \otimes V^\vee) = \Ext_{\GL(V)}^*(k, V \otimes V^\vee) \cong \Ext_{\GL(V)}^*(V, V),$$
and the latter has an associative $k$-algebra structure by the Yoneda product. We define
$$\Ext_{\GL}^*(I, I) := \lim_{V \in \mathsf{VIC}(k)^\mathrm{op}} \Ext_{\GL(V)}^*(V,V),$$
which again has an associative $k$-algebra structure. The limit here is taken over the opposite of the (essentially small) category $\mathsf{VIC}(k)$ from Section \ref{sec:CoeffSys}: a morphism $(f, U) : V \to W$ in this category gives an isomorphism $f \oplus inc : V \oplus U \overset{\sim}\to W$ and so induces a map
$$\Ext_{\GL(W)}^*(W,W) \overset{\text{restriction}}\lra \Ext_{\GL(V)}^*(V \oplus U,V \oplus U) \overset{\text{projection}}\lra \Ext_{\GL(V)}^*(V,V)$$
of associative $k$-algebras. In fact any automorphism $(f,0) : V \to V$ in $\mathsf{VIC}(k)$ acts on $\Ext_{\GL(V)}^*(V,V)$ as the identity---as an instance of the fact that inner automorphisms act trivially on group cohomology \cite[III (8.3)]{Brown}---so the limit is the same if we just take it over the standard split inclusions $0 \to k \to k^2 \to k^3 \to \cdots$. In each cohomological degree this limit is attained at a finite stage, by the stability results of Section \ref{sec:stab}.

\begin{theorem}\label{thm:BokstedtPeriodicity}
There are classes $x^{[i]} \in \Ext_{\GL}^{2i}(I, I)$ such that the map
$$\Gamma_{k}[x] = k\{x^{[0]}, x^{[1]}, x^{[2]}, \ldots\} \lra \Ext_{\GL}^{*}(I, I)$$
is an isomorphism of $k$-algebras from the free divided power algebra on $x = x^{[1]}$.
\end{theorem}

More generally, for finite sets $S$ and $T$ we define
$$\Ext_{\GL}^*(I^{\otimes S}, I^{\otimes T}) := \lim_{V \in \mathsf{VIC}(k)^\mathrm{op}} \Ext_{\GL(V)}^*(V^{\otimes S},V^{\otimes T})$$
analogously to the above: we wish to determine these groups. The Yoneda product and cup product extend to these, by their naturality. For a function $\ell : S \to \bN$, we can take cup products of the classes in Theorem \ref{thm:BokstedtPeriodicity} to obtain cohomology classes
$$\kappa(\ell) := \bigotimes_{s \in S} x^{[\ell(s)]} \in \Ext_{\GL}^{\sum_{s \in S} 2\ell(s)}(I^{\otimes S}, I^{\otimes S}).$$
Given in addition a bijection $f : T \to S$, the morphisms $V^{\otimes f} : V^{\otimes T} \to V^{\otimes S}$ assemble to an element $[f] \in \Ext_{\GL}^{0}(I^{\otimes T}, I^{\otimes S})$, and we can form the Yoneda product
$$\kappa(\ell, f) := \kappa(\ell) \circ [f] \in \Ext_{\GL}^{\sum_{s \in S} 2\ell(s)}(I^{\otimes T}, I^{\otimes S}).$$
Writing $\mathrm{Bij}(T,S)$ for the set of bijections from $T$ to $S$, and $k\{\mathrm{Bij}(T,S)\}$ for the free $k$-module on this set, this construction defines a map
\begin{equation}\label{eq:FF}
\Psi_{S,T}: \Gamma_{k}[x]^{\otimes S} \otimes k\{\mathrm{Bij}(T,S)\} \lra \Ext_{\GL}^{*}(I^{\otimes T}, I^{\otimes S}).
\end{equation}
We will often write  $\Gamma_{k}[x]^{\otimes S} = \Gamma_{k}[x_s \,|\, s \in S]$.

\begin{theorem}\label{thm:StabCohTensorPowers}
The map $\Psi_{S,T}$ is an isomorphism for all finite sets $S$ and $T$.
\end{theorem}

By the stability discussion in Section \ref{sec:stab} the map
$$\Ext_{\GL}^{*}(I^{\otimes T}, I^{\otimes S}) \lra \Ext_{\GL(V)}^{*}(V^{\otimes T}, V^{\otimes S}) \cong H^*(\GL(V) ; V^{\otimes S} \otimes (V^\vee)^{\otimes T})$$
is an isomorphism in a stable range of degrees, so Theorem \ref{thm:StabCohTensorPowers} determines the latter cohomology groups in a stable range.

\subsection{Proof of Theorems \ref{thm:BokstedtPeriodicity} and \ref{thm:StabCohTensorPowers}}
The proofs of these theorems are, at least implicitly, available in the literature on functor homology; we explain how to extract them, taking the paper of Franjou--Friedlander--Scorichenko--Suslin \cite{FFSS} as our main reference.

Let $\mathsf{F}$ denote the category of functors from finite-dimensional $k$-modules to $k$-modules, with $\Hom_\mathsf{F}(F,G)$ the $k$-module of natural transformations from $F$ to $G$. This is an abelian category with a set of projective generators, so one may do homological algebra in this category. In particular one may form $\Ext_\mathsf{F}^*(-,-)$ as the derived functors of $\Hom_\mathsf{F}(-,-)$.

For each $V \in \mathsf{VIC}(k)$ there is a functor
$$F \mapsto F(V) : \mathsf{F} \lra k[\GL(V)]\text{-}\mathsf{mod}$$
which is exact, so induces a map $\Ext_{\mathsf{F}}^*(F,G) \to \Ext_{\GL(V)}^{*}(F(V), G(V))$ for each $V$. Taking the limit over $\mathsf{VIC}(k)^\mathrm{op}$, these assemble into a map 
$$\Ext_{\mathsf{F}}^*(F,G) \lra \Ext_{\GL}^{*}(F, G),$$
which by \cite[Theorem A.1]{FFSS} or \cite{Betley} is an isomorphism.

Taking $F$ and $G$ to be the ``identity'' functor $I$, combining Théorème 7.3 and Section 11 of \cite{FLS} identifies $\Ext_\mathsf{F}(I,I)$, with its Yoneda product, with the divided power algebra $\Gamma_{k}[x] = k\{x^{[0]}, x^{[1]}, x^{[2]}, \ldots\}$. Combined with the previous paragraph this gives Theorem \ref{thm:BokstedtPeriodicity}.

\begin{remark}\label{rem:Conventionx}
The paper \cite{FLS} describes a specific choice of the generator $x$. For Part 2 of this paper, which concerns the case $k=\bF_p = \bZ/p$, it will be convenient for us to (perhaps) change this choice by a unit.

Consider the extension
\begin{equation}\label{eq:WittLift}
1 \lra \SL_n(\bZ/p^2, p) \lra \SL_n(\bZ/p^2) \lra \SL_n(\bZ/p) \lra 1,
\end{equation}
where the right-hand map is reduction modulo $p$, so the left-hand group consists of matrices over $\bZ/p^2$ of determinant 1 which reduce to the identity mod $p$. Elements of $\SL_n(\bZ/p^2, p)$ can be uniquely written in the form $I + p A$ for $A$ an $n \times n$ matrix with entries in $\bZ/p$ (using $p \cdot - : \bZ/p \overset{\sim}\to p\bZ/p^2$), and the identity $\det(I + p A) = 1 + p \cdot \mathrm{tr}(A) \in \bZ/p^2$ shows that $I+pA$ has determinant 1 if and only if $A$ has trace zero. This shows that the function $I + p A \mapsto A :  \SL_n(\bZ/p^2, p) \to sl_n(\bZ/p)$ is a bijection, and furthermore $(I+pA)(I+pB) = I+p(A+B)$, so this function is an isomorphism.

Under this identification, the class of the abelian extension \eqref{eq:WittLift} is an element $e_n \in H^2(\SL_n(\bZ/p) ; sl_n(\bZ/p))$. Writing $V_n = (\bZ/p)^n$, the exact sequence $0 \to sl_n(\bZ/p) \to V_n \otimes V_n^\vee \overset{ev}\to \bZ/p \to 0$ gives
$$\begin{tikzcd}[column sep=4pt]
   H^1(\SL_n(\bZ/p) ; \bF_p) \rar{\partial}
             \ar[draw=none]{d}[name=X, anchor=center]{}
    & H^2(\SL_n(\bZ/p) ; sl_n(\bZ/p)) \ar[rounded corners,
            to path={ -- ([xshift=2ex]\tikztostart.east)
                      |- (X.center) \tikztonodes
                      -| ([xshift=-2ex]\tikztotarget.west)
                      -- (\tikztotarget)}]{dl}[at end]{} \\      
   H^2(\SL_n(\bZ/p) ; V_n \otimes V_n^\vee) \rar{ev_*} & H^2(\SL_n(\bZ/p) ; \bZ/p)
\end{tikzcd}$$
and the outer terms vanish (for $n$ large enough) by the theorem of Quillen \cite{QuillenFiniteFields} so the middle map is an isomorphism. We define $x'_n \in H^2(\SL_n(\bZ/p) ; V_n \otimes V_n^\vee)$ to be the image of \emph{minus} the class $e_n$. One may check that these classes are compatible under stabilisation so give an $x' \in \Ext^2_{\GL}(I, I)$. This class is not zero. If it were, then $e_n$ would vanish for $n$ large enough so the extensions \eqref{eq:WittLift} would be split, but they are not: see \cite[Proposition 0.3]{SahII}. Thus $x' = u \cdot x$ for some unit $u$ in $\bZ/p$, so we can let $(x')^{[i]} := u^i \cdot x^{[i]}$ so that $\Ext^*_{\GL}(I, I) = \Gamma_{\bF_p}[x']$. In Theorem \ref{thm:BokstedtPeriodicity} for $k=\bF_p$ we take this $x'$ for $x$. (A similar normalisation can be made for any finite field $k$, substituting the length 2 Witt vectors $W_2(k)$ for $\bZ/p^2$ in the above discussion.)
\end{remark}

If $|S| \neq |T|$ then Pirashvili's cancellation lemma \cite[Theorem A.1]{BetleyPirashvili} shows that $\Ext_{\mathsf{F}}^*(I^{\otimes T}, I^{\otimes S}) \cong \Ext_{\GL}^{*}(I^{\otimes T}, I^{\otimes S})$ vanishes. If $|S|=|T|$ then \cite[Corollary 1.8]{FFSS}, using Pirashvili's cancellation lemma to neglect most terms, gives an isomorphism
$$\Ext_{\mathsf{F}}^*(I,I)^{\otimes S} \otimes k\{\mathrm{Bij}(T,S)\} \cong \Ext_{\mathsf{F}}^*(I^{\otimes T}, I^{\otimes S}).$$
Combined with the previous two paragraphs, and after checking that the maps which induce this isomorphism agree with those that we have described above, this gives Theorem \ref{thm:StabCohTensorPowers}.

\section{The walled Brauer category}\label{sec:wBr}

A useful bookkeeping device for keeping track of the groups $\Ext_{\GL}^{*}(I^{\otimes T}, I^{\otimes S})$, or the groups $H^*(\GL(V) ; V^{\otimes S} \otimes (V^\vee)^{\otimes T})$, and the various maps between them induced by bijections $S \overset{\sim}\to S'$ or $T' \overset{\sim}\to T$ or (co)evaluations, is to consider the totality of these groups as forming a representation of the \emph{(upward) walled Brauer category}. We define these below, and will only make use of their formal definition.

\subsection{Functoriality on the upward walled Brauer category}\label{sec:uwBr}

The \emph{upward walled Brauer category} $\mathsf{uwBr}$ is the category with objects given by pairs $(S,T)$ of finite sets, and with morphisms $\mathsf{uwBr}((S,T), (U,V))$ given by a pair of injections $f : S \to U$, $g: T \to V$ and well as a bijection $m : U \setminus f(S) \to V \setminus g(T)$. We visualise such morphisms as in the figure below, where the composition is given by gluing such 1-dimensional cobordisms.

\begin{figure}[h]
\begin{center}
\includegraphics[width=6cm]{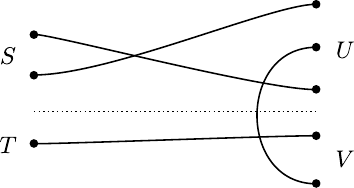}
\end{center}
\end{figure}

The construction $(S, T) \mapsto \Ext_{\GL}^{*}(I^{\otimes T}, I^{\otimes S})$ defines a functor from $\mathsf{uwBr}$ to graded $k$-modules:  a morphism $(f,g,m)$ acts as
$$\Ext_{\GL}^{*}(I^{\otimes T}, I^{\otimes S}) \overset{(f,g)}\lra \Ext_{\GL}^{*}(I^{\otimes g(T)}, I^{\otimes f(S)}) \overset{m}\lra \Ext_{\GL}^{*}(I^{\otimes V}, I^{\otimes U}),$$
where the first map is induced by the bijections $f : S \to f(S)$ and $g : T \to g(T)$, and the second is given by inserting $1 \in \Ext_{\GL}^{0}(I, I)$ along each of the pairs $\{(u, m(u))\}_{u \in U \setminus f(S)}$.

Similarly, the construction $(S,T) \mapsto \Gamma_k[x]^{\otimes S} \otimes k\{\mathrm{Bij}(T,S)\}$ defines a functor from $\mathsf{uwBr}$ to graded $k$-modules: a morphism $(f,g,m)$ acts by sending the element $(\bigotimes_{s \in S} x_s^{[\ell(s)]}) \otimes \phi$ to $(\bigotimes_{u \in U} x_u^{[\ell'(u)]}) \otimes \phi'$ where $\ell'(u)$ is $\ell(s)$ if $u=f(s)$ and is 0 otherwise, and the bijection $\phi' : V \to U$ is equal to $f\circ\phi\circ g^{-1}$ on $g(T) \subset V$ and is equal to $m^{-1}$ on $V \setminus g(T)$.

As $x^{[0]} = 1 \in \Ext_{\GL}^{0}(I, I)$, it follows from these descriptions that the map \eqref{eq:FF} we have described is a natural transformation
$$\Psi : \Gamma_k[x]^{\otimes -} \otimes k\{\mathrm{Bij}(\bullet,-)\} \Longrightarrow \Ext_{\GL}^{*}(I^{\otimes \bullet}, I^{\otimes -}) $$
of functors from $\mathsf{uwBr}$ to graded $k$-modules. Theorem \ref{thm:StabCohTensorPowers} then says that it is in fact a natural isomorphism of such functors.

\subsection{Functoriality on the full walled Brauer category}\label{sec:FullWBr}

The discussion in this section is not needed for the proof of Theorem \ref{thm:mainFin}, but will be used in the proof of Theorem \ref{thm:mainCong}.

For $\delta \in k$ the \emph{walled Brauer category} $\mathsf{wBr}_\delta$ is the $k$-linear category with objects given by pairs $(S,T)$ of finite sets, and with morphisms $\mathsf{wBr}_\delta((S,T), (U,V))$ given by the $k$-vector space with basis given by tuples $(f, g, m, n)$ where $f : S' \to U'$ is a bijection from a subset $S' \subset S$ to a subset $U' \subset U$, $g : T' \to V'$ is a bijection from a subset $T' \subset T$ to a subset $V' \subset V$, and $m : S \setminus S' \to T \setminus T'$ and $n : U \setminus U' \to V \setminus V'$ are bijections. We depict such morphisms as in the figure below, where as shown the composition is given by gluing such 1-dimensional cobordisms, and replacing any circles that are formed by the scalar $\delta \in k$.

\begin{figure}[h]
\begin{center}
\includegraphics[width=12cm]{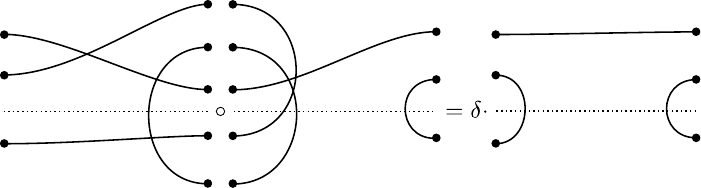}
\end{center}
\end{figure}

For $V$ a \emph{finite-dimensional} vector space, the functor
$$(S,T) \longmapsto H^*(\GL(V) ; V^{\otimes S} \otimes (V^\vee)^{\otimes T})$$
from $\mathsf{uwBr}$ to graded $k$-modules extends to a functor on $\mathsf{wBr}_\delta$ with $\delta := \dim(V)$. This extension is induced by the evaluation maps $ev: V \otimes V^\vee \to k$. Concretely, the morphism
$$(inc : S \setminus s \to S, inc : T \setminus t \to T, \emptyset \overset{\sim}\to\emptyset, \{s\} \overset{\sim}\to \{t\}) : (S,T) \lra (S \setminus s, T \setminus t)$$
gives the map induced on cohomology by the morphism of coefficient systems 
$$\epsilon_{s,t} : V^{\otimes S } \otimes (V^\vee)^{\otimes T} \lra V^{\otimes S \setminus s} \otimes (V^\vee)^{\otimes T \setminus t}$$ which evaluates the $s$th and $t$th terms. Using that the composition $k \overset{coev}\to V \otimes V^\vee \overset{ev}\to k$ is multiplication by $\dim(V)=\delta$, it is elementary to verify that there is a unique extension to a functor on $\mathsf{wBr}_\delta$ given in this way on such morphisms.

Composing the maps $\Psi_{S,T}$ from \eqref{eq:FF} with evaluation at $V$ gives maps
$$\psi_{S,T} : \Gamma_{k}[x]^{\otimes S} \otimes k\{\mathrm{Bij}(T,S)\} \lra H^*(\GL(V) ; V^{\otimes S} \otimes (V^\vee)^{\otimes T}),$$
which are isomorphisms in a stable range by Theorem \ref{thm:StabCohTensorPowers}, and which are natural transformations of functors on $\mathsf{uwBr}$. We wish to explain how the $\mathsf{wBr}_\delta$-functoriality of the target translates to the source.

Define a map
$$
\delta_{s,t} : \Gamma_{k}[x]^{\otimes S} \otimes k\{\mathrm{Bij}(T, S)\} \lra \Gamma_{k}[x]^{\otimes S \setminus s } \otimes k\{\mathrm{Bij}(T \setminus t, S \setminus s)\}$$
by the formula
$$\big(\bigotimes_{u \in S} x_u^{[\ell(u)]}\big) \otimes \sigma  \longmapsto \begin{cases}
\dim(V) \big(\bigotimes_{u \in S \setminus s} x^{[\ell(u)]}_u \big) \otimes \sigma\vert_{T \setminus t} & \sigma(t)=s \text{ and } \ell(s)=0\\
0 & \sigma(t)=s \text{ and } \ell(s) > 0\\
\binom{\ell(s)+ \ell(\sigma(t))}{\ell(s)} \big(\bigotimes_{u \in S \setminus s} x_u^{[\ell'(u)]} \big)\otimes \sigma' & \sigma(t) \neq s
\end{cases}$$
where $\ell'$ is given by $\ell'(\sigma(t)) = \ell(s)+ \ell(\sigma(t))$ and by the restriction of $\ell$ on all other elements of $S \setminus s$, and $\sigma'$ is given by $\sigma'(\sigma^{-1}(s)) = \sigma(t)$, and by the restriction of $\sigma$ on all other elements of $T \setminus t$.

\begin{lemma}\label{lem:ev}
The square
\begin{equation*}
\begin{tikzcd}
\Gamma_{k}[x]^{\otimes S} \otimes k\{\mathrm{Bij}(T,S)\} \dar{\psi_{S,T}}\rar{\delta_{s,t}} & \Gamma_{k}[x]^{\otimes S \setminus s } \otimes k\{\mathrm{Bij}(T \setminus t , S \setminus s )\} \dar{\psi_{S \setminus s,T \setminus t}}\\
H^*(\GL(V) ; V^{\otimes S } \otimes (V^\vee)^{\otimes T}) \rar{\epsilon_{s,t}}
 & H^*(\GL(V) ; V^{\otimes S \setminus s } \otimes (V^\vee)^{\otimes T \setminus t })
\end{tikzcd}
\end{equation*}
 commutes in a stable range of degrees.
\end{lemma}
\begin{proof}
If $\sigma(t)=s$ then, by taking tensor products, we can reduce to the case $S=\{s\}$ and $T=\{t\}$. In this case the result is given by applying the evaluation map to $x^{[\ell(s)]} \in H^{2\ell(s)}(\GL(V) ; V \otimes V^\vee)$. If $\ell(s)>0$ then the result of this evaluation map is 0, as in the stable range $H^{2\ell(s)}(\GL(V) ; k)=0$ in that case. If $\ell(s)=0$ then the element $x^{[0]} \in H^0(\GL(V) ; V \otimes V^\vee)$ is the coevaluation, so applying the evaluation map to it gives $\dim(V) = \delta \in k = H^0(\GL(V) ; k)$.

If $\sigma(t) = s' \neq s$, then we can write $t' := \sigma^{-1}(s) \neq t$ and again by taking tensor products we reduce to the case $S=\{s,s'\}$ and $T=\{t,t'\}$. Then
\begin{equation*}
\psi_{S,T}(x_s^{[\ell(s)]} x_{s'}^{[\ell(s')]} \otimes \sigma) \in H^*(\GL(V) ; V^{\otimes S } \otimes (V^\vee)^{\otimes T})
\end{equation*}
is the cup product of the classes
\begin{align*}
\psi_{\{s\}, \{t'\}}(x_s^{[\ell(s)]}) &\in H^*(\GL(V) ; V^{\otimes \{s\}} \otimes (V^\vee)^{\otimes\{t'\}}) \cong \Ext^*_{\GL(V)}(V, V)\\
\psi_{\{s'\}, \{t\}}(x_{s'}^{[\ell(s')]}) &\in H^*(\GL(V) ; V^{\otimes \{s'\}} \otimes (V^\vee)^{\otimes\{t\}}) \cong \Ext^*_{\GL(V)}(V, V),
\end{align*}
and applying $\epsilon_{s,t}$ corresponds to evaluating the Yoneda product. By the divided power algebra structure described in Theorem \ref{thm:BokstedtPeriodicity}, the result is
$$\binom{\ell(s)+ \ell(s')}{\ell(s)} x^{[\ell(s)+ \ell(s')]},$$
which agrees with $\psi_{\{s'\}, \{t'\}} \delta_{s,t}$ applied to $x_s^{[\ell(s)]} x_{s'}^{[\ell(s')]} \otimes \sigma$.
\end{proof}

\subsection{Graphical interpretation}\label{sec:GraphicalInt}

\begin{wrapfigure}[9]{r}{0.3\linewidth}
\vspace{-2ex}
\centering
\includegraphics[width=3cm]{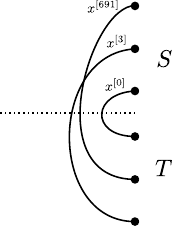}\\
\end{wrapfigure}

Rather than the formulas given above, we can interpret the functoriality of $\Gamma_{k}[x]^{\otimes -} \otimes k\{\mathrm{Bij}(\bullet,-)\}$ on the walled Brauer category by interpreting elements of $\Gamma_{k}[x]^{\otimes S} \otimes k\{\mathrm{Bij}(T,S)\}$ as given graphically as shown to the right. That is, an element of $\mathsf{wBr}_\delta((\emptyset, \emptyset), (S,T))$ with each strand labelled by an $x^{[i]}$.

Then the functoriality is given by concatenating with an element of the walled Brauer category, multiplying labels which now lie on the same strand together using the divided power multiplication,  then setting any closed components labelled by $x^{[i]}$ with $i>0$ equal to zero, and setting any closed components labelled by $x^{[0]}$ equal to $\dim(V)$.

\section{Proof of Theorem \ref{thm:mainFin}: tensor powers}

In this section we prove the following variant of Theorem \ref{thm:mainFin}: with $\SL(V)$ replaced by $\GL(V)$, with the representations $S_{\lambda, \mu}(V)$ replaced by $V^{\otimes n} \otimes (V^\vee)^{\otimes m}$, and without conditions on the size of $n$ and $m$. That is, the statement that
\begin{equation}\label{eq:TensorPowerCase}
H^i(\GL(V) ; V^{\otimes n} \otimes (V^\vee)^{\otimes m}) \cong [V^{\otimes n} \otimes (V^\vee)^{\otimes m} \otimes X^i]^{\GL(V)}
\end{equation}
for $i < 2p$, as long as $\dim(V)$ is large enough. In this statement $\GL(V)$ can be replaced by $\SL(V)$ by the same argument as Corollary \ref{cor:GLvSL}. In Section \ref{sec:GeneralCase} we will explain how to deduce from this the statement of Theorem \ref{thm:mainFin} for the $S_{\lambda, \mu}(V)$: it is there that the conditions on $n$ and $m$ will arise. Our proof of \eqref{eq:TensorPowerCase} will be by calculating both sides and comparing them. Given the homological stability results of Section \ref{sec:Stability}, the left-hand side has been calculated by Theorem \ref{thm:StabCohTensorPowers}. The main task of this section is therefore to calculate the right-hand side.

Following the notation used in Section \ref{sec:FunctorHomology}, for a functor $F$ from finite-dimensional $k$-modules to $k$-modules it is convenient to define
$$H^i(\GL ; F) := \lim_{V \in \mathsf{VIC}(k)^\mathrm{op}} H^i(\GL(V) ; F(V)).$$
(It then agrees with $\Ext_{\GL}^i(k, F)$.) Theorem \ref{thm:StabCohTensorPowers} provides isomorphisms
$$ k\{ \mathrm{Bij}(T,S)\} \overset{\sim}\lra H^{0}(\GL ; I^{\otimes S} \otimes (I^\vee)^{\otimes T})$$
given by inserting copies of $coev: k \to I \otimes I^\vee$ and permuting the $I^\vee$ terms. Let $S=T= \ul{r} := \{1,2,\ldots,r\}$, identify $I^{\otimes S} \otimes (I^\vee)^{\otimes T} = (I \otimes I^\vee)^{\otimes r}$, so that the above gives an isomorphism 
$$ k\{\Sigma_r\}\overset{\sim}\lra H^{0}(\GL ; (I \otimes I^\vee)^{\otimes r})$$
where the $\Sigma_r$-action on the right-hand side by permuting the tensor factors corresponds on the left-hand side to the action of $\Sigma_r$ on itself by conjugation; to avoid confusion we write $\Sigma_r^{ad}$ for this $\Sigma_r$-set. One should visualise elements of $\Sigma_r^{ad}$ as permutations presented as disjoint cycles.

We write $\mathrm{Sym}^r(I \otimes I^\vee) = ((I \otimes I^\vee)^{\otimes r})_{\Sigma_r}$ for the coinvariants of the action which permutes the tensor factors. Commuting $H^0(\GL ; -)$ with $(-)_{\Sigma_r}$ defines a map
$$k\{\Sigma_r^{ad}\}_{\Sigma_r}\overset{\sim}\lra H^{0}(\GL ; (I \otimes I^\vee)^{\otimes r})_{\Sigma_r} \lra  H^{0}(\GL ; \mathrm{Sym}^r(I \otimes I^\vee)).$$
The conjugacy class of $r$-cycles in $\Sigma_r$ gives a well-defined element in the left-hand term, which under this map defines an element
$$c_r \in  H^{0}(\GL ; \mathrm{Sym}^r(I \otimes I^\vee)).$$

\begin{lemma}\label{lem:invariants}
The map
$$k[c_1, c_2, \ldots] \lra H^0(\GL ; \mathrm{Sym}^\bullet(I \otimes I^\vee))$$
is an isomorphism of graded $k$-algebras in gradings $\bullet < p$.
\end{lemma}
\begin{proof}
For $r < p$ taking $\Sigma_r$-coinvariants is exact (because then $|\Sigma_r|=r!$ is invertible in $k$), giving an isomorphism
$$k\{\Sigma_r^{ad}\}_{\Sigma_r} \overset{\sim}\lra H^{0}(\GL ; \mathrm{Sym}^r(I \otimes I^\vee)).$$
For $\sum_i i \cdot a_i = r$ the image of the monomial $c_1^{a_1} c_2^{a_2} \cdots c_r^{a_r} \in k[c_1, c_2, c_3, \ldots]$ in $H^0(\GL ; \mathrm{Sym}^r(I \otimes I^\vee)) \cong k\{\Sigma_r^{ad}\}_{\Sigma_r}$ is the class of any permutation having precisely $a_i$-many $i$-cycles. This is visibly a bijection.
\end{proof}

As in the introduction, define a graded ring object $X^\bullet$ in coefficient systems by
\begin{equation}\label{eq:SymModRegSeq}
\mathrm{Sym}^\bullet(I \otimes I^\vee)/(c_1, c_2, \ldots)
\end{equation}
with grading doubled, so that $c_i$ has degree $2i$. This is to be constructed, and interpreted, as follows. For every vector space $V$ the class $c_i$ defines a class $c_i^V \in \mathrm{Sym}^\bullet(V \otimes V^\vee)$ via
$$H^0(\GL ; \mathrm{Sym}^\bullet(I \otimes I^\vee)) \lra H^0(\GL(V) ; \mathrm{Sym}^\bullet(V \otimes V^\vee)) \subset \mathrm{Sym}^\bullet(V \otimes V^\vee),$$
and we can form the quotient graded $k$-algebra
$$\mathrm{Sym}^\bullet(V \otimes V^\vee)/(c_1^V, c_2^V, \ldots, c_i^V).$$
These assemble into a graded coefficient system $\mathrm{Sym}^\bullet(I \otimes I^\vee)/(c_1, c_2, \ldots, c_i)$, and we write $X^\bullet_i$ for this coefficient system with its grading doubled. There are natural maps between these, and we define \eqref{eq:SymModRegSeq} by $\colim_{i \to \infty} \mathrm{Sym}^\bullet(I \otimes I^\vee)/(c_1, c_2, \ldots, c_i)$, and write $X^\bullet = X^\bullet_\infty = \colim_{i \to \infty} X^\bullet_i$.

The following two lemmas are somewhat technical, but will be used to justify the claim (in Corollary \ref{cor:InsideOutside}) that we may commute $H^*(\GL ; -)$ with quotienting by the $c_i$, at least in gradings $\bullet < p$.

\begin{lemma}\label{lem:RegOnSym}
There are sequences of graded coefficient systems
\begin{equation*}
\begin{tikzcd}[row sep=small]
0 \lra  \frac{\mathrm{Sym}^\bullet(I \otimes I^\vee)}{(c_1,  \ldots, c_{i-1})}[i] \rar{c_{i}}& \frac{\mathrm{Sym}^\bullet(I \otimes I^\vee)}{(c_1,  \ldots, c_{i-1})} \rar & \frac{\mathrm{Sym}^\bullet(I \otimes I^\vee)}{(c_1,  \ldots, c_{i})} \lra 0 
\end{tikzcd}
\end{equation*}
which are exact in gradings $\bullet < p$ when evaluated on $k$-modules $V$ of large dimension.
\end{lemma}
\begin{proof}
For any $k$-module $V$ the composition
\begin{equation}\label{eq:Specialise}
k[c_1, c_2, \ldots] \to H^0(\GL ; \mathrm{Sym}^\bullet(I \otimes I^\vee)) \to H^0(\GL(V) ; \mathrm{Sym}^\bullet(V \otimes V^\vee))
\end{equation}
sends, by definition, $c_i$ to $c_i^V \in \mathrm{Sym}^i(V \otimes V^\vee)$, so it suffices to show that $c_1^V, \ldots, c_{p-1}^V$ is a regular sequence in the graded $k$-algebra $\mathrm{Sym}^\bullet(V \otimes V^\vee)$ for all large enough $\dim(V)$. Recall that a sequence $a_1, \ldots, a_n$ of elements in a commutative algebra $A$ is \emph{regular} if each $a_i$ is not a zerodivisor in $A/(a_1, \ldots, a_{i-1})$. This may be tested after base change to an algebraic closure of $k$, which we now implicitly do.

We identify $V^\vee \otimes V$ with $\mathrm{End}(V)$ considered as an affine algebraic variety, and so identify the graded $k$-algebra $\mathrm{Sym}^\bullet(V \otimes V^\vee)$ with the ring of homogeneous regular functions on $\mathrm{End}(V)$. There are homogeneous regular functions $\sigma_1, \sigma_2, \ldots, \sigma_{\dim(V)}$ given by the coefficients of the characteristic polynomial $\det(tI - A) = t^{\dim(V)} + \sigma_1(A) t^{\dim(V)-1} + \cdots + \sigma_{\dim(V)}(A)$. Define $\sigma_0(A) = 1$ too. These functions are also $\GL(V)$-invariant, and we first wish to relate them to the $c_i^V$.

Choosing a basis to identify $V = k^n$, there are matrix coordinate functions $x_{i,j} : \mathrm{End}(k^n) \to k$ for $i, j \in \{1,\ldots,n\}$, and $\mathrm{Sym}^\bullet(V \otimes V^\vee)$ is the polynomial algebra on these. Spelling out our definition of $c_r^V$ in terms of these functions gives
$$c_r^V = \sum_{(i_1, \ldots, i_r) \in \{1,\ldots,n\}^r} x_{i_1, i_2} x_{i_2, i_3} \cdots x_{i_r, i_1}.$$
Evaluated at a generic diagonal matrix $A = \mathrm{diag}(\lambda_1, \ldots, \lambda_n) \in \mathrm{End}(k^n)$, this sum only has non-zero contributions when $i_1=i_2 = \cdots = i_r$, giving
$$c_r^V(A) = \sum_{i=1}^n \lambda_i^r,$$
i.e.\ it is the $r$th power-sum symmetric polynomial $p_r(\lambda_1, \ldots, \lambda_n)$ in the eigenvalues of $A$. In other words, it is $\mathrm{Tr}(A^r)$. On the other hand
$$t^{\dim(V)} + \sigma_1(A) t^{\dim(V)-1} + \cdots + \sigma_{n}(A) = \det(tI - A) = \prod_{i=1}^n(t-\lambda_i),$$
i.e.\ $\sigma_i(A)$ is $(-1)^i$ times the $i$th elementary symmetric polynomial $e_i(\lambda_1, \ldots, \lambda_n)$ in the eigenvalues of $A$. By the Girard--Newton identities, for $r \leq n$ the identity
$$r  \sigma_r + \sum_{i=1}^r  \sigma_{r-i} c^V_i = 0$$
holds on the locus of diagonal matrices: as the left-hand side is $\GL(V)$-invariant the identity also holds on the locus of diagonalisable matrices, and as this locus is Zariski-dense the identity holds in $\mathrm{Sym}^\bullet(V \otimes V^\vee)$. From these identities it follows that the sequence $c^V_1, \ldots, c^V_{p-1}$ is regular if and only if the sequence $\sigma_1, \ldots, \sigma_{p-1}$ is.

The homogeneous ideal $(\sigma_1, \ldots, \sigma_{\dim(V)})$ defines the subvariety of $\mathrm{End}(V)$ consisting of those endomorphisms with characteristic polynomial $t^{\dim(V)}$, i.e.\ the nilpotent endomorphisms. This subvariety is well-known to have codimension $\dim(V) = \mathrm{rk}(\GL(V))$, see e.g. \cite[p.\ 64]{JantzenNilpotent}. As the sequence $\sigma_1, \ldots, \sigma_{\dim(V)}$ consists of homogeneous elements in the graded polynomial ring $\mathrm{Sym}^\bullet(V \otimes V^\vee)$, and generates an ideal of codimension $\dim(V)$, it follows that it is a regular sequence (see e.g.\ Proposition 4.3.4 of \cite{Benson}, paying attention to Hypothesis 4.3.2 (b) and using that finitely generated polynomial rings are Cohen--Macaulay). Thus the subsequence $\sigma_1, \ldots, \sigma_{p-1}$ is also regular.
\end{proof}

\begin{remark}
This discussion also shows that $c_p = c_1^p \in H^0(\GL ; \mathrm{Sym}^\bullet(I \otimes I^\vee))$, so Lemma \ref{lem:invariants} is sharp.
\end{remark}

Using the map $k[c_1, c_2,  \ldots] \to H^0(\GL ; \mathrm{Sym}^\bullet(I \otimes I^\vee))$ and cup product makes $H^*(\GL ; I^{\otimes S} \otimes (I^\vee)^{\otimes T} \otimes \mathrm{Sym}^\bullet(I \otimes I^\vee))$ into a (right) $k[c_1, c_2,  \ldots]$-module.

\begin{lemma}\label{lem:RegOnCohOfSym}
The kernel of multiplication by $c_i$ on
$$H^*(\GL ; I^{\otimes S} \otimes (I^\vee)^{\otimes T} \otimes \mathrm{Sym}^\bullet(I \otimes I^\vee))/(c_1, \ldots, c_{i-1})$$
is trivial in gradings $\bullet < p$.
\end{lemma}
\begin{proof}
For $r < p$ taking $\Sigma_r$-coinvariants is exact. Thus by Theorem \ref{thm:StabCohTensorPowers} the cohomology of $I^{\otimes S} \otimes (I^\vee)^{\otimes T} \otimes \mathrm{Sym}^r(I \otimes I^\vee)$ with $r<p$ is identified with the $\Sigma_r$-coinvariants of
$$k\{x^{[0]}, x^{[1]}, x^{[2]}, \ldots\}^{\otimes S \sqcup \ul{r}} \otimes k\{\mathrm{Bij}(T \sqcup \ul{r}, S \sqcup \ul{r})\}.$$
This is the vector space with basis the set $\bN^{S \sqcup \ul{r}} \times \mathrm{Bij}(T \sqcup \ul{r}, S \sqcup \ul{r})$, so the coinvariants are identified with the vector space with basis the $\Sigma_r$-orbits of this set. 

In this picture, multiplication by $c_i \in H^0(\GL ; \mathrm{Sym}^i(I \otimes I^\vee))$ corresponds to 
\begin{align*}
&k\{(\bN^{S \sqcup \ul{r-i}} \times \mathrm{Bij}(T \sqcup \ul{r-i}, S \sqcup \ul{r-i}))/\Sigma_{r-i}\}\\
& \quad\quad\quad\quad\quad\quad \lra k\{(\bN^{S \sqcup \ul{r}} \times \mathrm{Bij}(T \sqcup \ul{r}, S \sqcup \ul{r}))/\Sigma_r\}
\end{align*}
which adjoins the $i$-cycle $(r-i+1, r-i+2, \ldots, r)$ with all labels $0 \in \bN$. Dividing out by $(c_1, \ldots, c_{i-1})$ means quotienting by the subspace of those elements which contain a $j$-cycle with all labels $0 \in \bN$ and $j<i$. Adjoining the $i$-cycle $(r-i+1, r-i+2, \ldots, r)$ with all labels $0 \in \bN$ cannot cause there to be such $j$-cycles if they were not already present, which proves the claim.
\end{proof}

\begin{corollary}\label{cor:InsideOutside}
The natural map 
$$H^*(\GL ; I^{\otimes S} \otimes (I^\vee)^{\otimes T} \otimes \mathrm{Sym}^\bullet(I \otimes I^\vee))/(c_1, c_2, \ldots) \lra H^*(\GL ; I^{\otimes S} \otimes (I^\vee)^{\otimes T} \otimes X^{2\bullet})$$
is an isomorphism in gradings $\bullet < p$.
\end{corollary}
\begin{proof}
We will show by induction over $i$ that the natural map
$$ H^*(\GL ; I^{\otimes S} \otimes (I^\vee)^{\otimes T} \otimes\mathrm{Sym}^\bullet(I \otimes I^\vee))/(c_1, c_2, \ldots, c_i) \to H^*(\GL ; I^{\otimes S} \otimes (I^\vee)^{\otimes T} \otimes X^{2\bullet}_i)$$
is an isomorphism for $\bullet < p$; it is tautological for $i=0$. By doubling the grading of the sequence in Lemma \ref{lem:RegOnSym} there are sequences of graded coefficient systems
$$ X_{i-1}^{2\bullet}[2i] \overset{c_{i+1}}\lra X_{i-1}^{2\bullet} \lra X_{i}^{2\bullet}$$
which are exact for $\bullet < p$ when evaluated on all $k$-modules $V$ of large enough dimension. This gives a long exact sequence on stable cohomology
$$\begin{tikzcd}[column sep=8pt]
   H^*(\GL ; I^{\otimes S} \otimes (I^\vee)^{\otimes T} \otimes X^{2\bullet}_{i-1}[2i]) \rar{- \cdot c_{i}}
             \ar[draw=none]{d}[name=X, anchor=center]{}
    & H^*(\GL ; I^{\otimes S} \otimes (I^\vee)^{\otimes T} \otimes X^{2\bullet}_{i-1}) \ar[rounded corners,
            to path={ -- ([xshift=2ex]\tikztostart.east)
                      |- (X.center) \tikztonodes
                      -| ([xshift=-2ex]\tikztotarget.west)
                      -- (\tikztotarget)}]{dl}[at end]{} \\      
   H^*(\GL ; I^{\otimes S} \otimes (I^\vee)^{\otimes T} \otimes X^{2\bullet}_{i}) \rar{\partial} 
   		\ar[draw=none]{d}[name=Y, anchor=center]{}
   & H^{*+1}(\GL ; I^{\otimes S} \otimes (I^\vee)^{\otimes T} \otimes X^{2\bullet}_{i-1}[2i]) \ar[rounded corners,
            to path={ -- ([xshift=2ex]\tikztostart.east)
                      |- (Y.center) \tikztonodes
                      -| ([xshift=-2ex]\tikztotarget.west)
                      -- (\tikztotarget)}]{dl}[swap, near end]{- \cdot c_{i}} \\  
    H^{*+1}(\GL ; I^{\otimes S} \otimes (I^\vee)^{\otimes T} \otimes X^{2\bullet}_{i-1}) \rar & \cdots.   
\end{tikzcd}$$
By the inductive assumption and Lemma \ref{lem:RegOnCohOfSym}, the maps $- \cdot c_{i}$ are injective in gradings $\bullet < p$, so the maps $\partial$ are trivial in this range. This shows that
$$H^*(\GL ; I^{\otimes S} \otimes (I^\vee)^{\otimes T} \otimes X_{i}^{2\bullet}) = H^*(\GL ; I^{\otimes S} \otimes (I^\vee)^{\otimes T} \otimes X_{i-1}^{2\bullet})/(c_{i}),$$
which with the inductive assumption again finishes the induction.

Finally, the map $X^{2\bullet}_p \to X^{2\bullet}_\infty = X^{2\bullet}$ is an isomorphism in gradings $\bullet < p$, so we in fact obtain the claimed statement at a finite stage of the induction.
\end{proof}

The class of an $(r+1)$-cycle in $k\{\Sigma^{ad}_{1+r}\}$ gives an element
$$\bar{d}_r \in k\{\Sigma^{ad}_{1+r}\}_{\Sigma_r} \lra H^0(\GL ; I \otimes I^\vee \otimes \mathrm{Sym}^r(I \otimes I^\vee)),$$
independent of the choice of $(r+1)$-cycle, and hence in the quotient an element
$$d_r \in H^0(\GL ; I \otimes I^\vee \otimes X^{2r}).$$

Using the $k$-algebra structure of $X^\bullet$ there is an induced map
$$k\{d_0, d_1, d_2, \ldots\}^{\otimes S} \to H^0(\GL ; I^{\otimes S} \otimes (I^\vee)^{\otimes S} \otimes X^{\bullet \otimes S}) \to H^0(\GL ; I^{\otimes S} \otimes (I^\vee)^{\otimes S} \otimes X^\bullet),$$
and then acting on the $I^\vee$'s by bijections gives a map
\begin{equation}\label{eq:OtherForm}
k\{d_0, d_1, d_2, \ldots\}^{\otimes S} \otimes k\{\mathrm{Bij}(T,S)\} \lra H^0(\GL ; I^{\otimes S} \otimes (I^\vee)^{\otimes T} \otimes X^\bullet).
\end{equation}
This is in fact a natural transformation of functors from the upwards walled Brauer category to $k$-modules, as in Section \ref{sec:uwBr}.

\begin{lemma}
The map \eqref{eq:OtherForm} is an isomorphism for $\bullet < 2p$ and all  $S$ and $T$.
\end{lemma}
\begin{proof}
By Corollary \ref{cor:InsideOutside} we have an identification
$$H^0(\GL ; I^{\otimes S} \otimes (I^\vee)^{\otimes T} \otimes \mathrm{Sym}^\bullet(I \otimes I^\vee))/(c_1, c_2, \ldots) \overset{\sim}\lra H^0(\GL ; I^{\otimes S} \otimes (I^\vee)^{\otimes T} \otimes X^{2\bullet})$$
for $\bullet < p$. For $r< p$ the object $\mathrm{Sym}^r(I \otimes I^\vee)$ is a summand of $(I \otimes I^\vee)^{\otimes r}$, so the left-hand side is a subquotient of $H^0(\GL ; I^{\otimes S \sqcup \ul{r}} \otimes (I^\vee)^{\otimes T \sqcup \ul{r}})$, and hence by Theorem \ref{thm:StabCohTensorPowers} it vanishes unless $|S|=|T|$.

We therefore choose a bijection $S \overset{\sim}\to T$. As in the proof of Lemma \ref{lem:RegOnCohOfSym}, the map
$$k\{\Sigma_{S \sqcup \ul{r}}^{ad}\} \lra H^0(\GL ; I^{\otimes S \sqcup \ul{r}} \otimes (I^\vee)^{\otimes S \sqcup \ul{r}}),$$
given by acting on the coevaluation element by permuting the $I^\vee$'s, is an isomorphism. It remains an isomorphism on taking coinvariants $(-)_{\Sigma_r}$ and, as this is an exact functor when $r<p$, these coinvariants can be commuted with $H^0(\GL ; -)$ to see that the induced map
$$k\{\Sigma_{S \sqcup \ul{r}}^{ad}/\Sigma_r\} \lra H^0(\GL ; I^{\otimes S} \otimes (I^\vee)^{\otimes S} \otimes \mathrm{Sym}^r(I \otimes I^\vee))$$
is an isomorphism as long as $r < p$. Multiplying by $c_i \in H^0(\GL ; \mathrm{Sym}^i(I \otimes I^\vee))$ on the right-hand side translates on the left-hand side to the map $\Sigma_{S \sqcup \ul{r-i}}^{ad}/\Sigma_{r-i} \to \Sigma_{S \sqcup \ul{r}}^{ad}/\Sigma_{r}$ which adds an $i$-cycle of elements in $\ul{r}$. Thus the quotient by the $c_i$ on the right-hand side translates on the left-hand side to killing those basis elements which are represented by a permutation having a cycle of elements in $\ul{r}$. Thus what remains are the permutations of $S \sqcup \ul{r}$ where every cycle contains an element of $S$ (let us call this $F_r \subset \Sigma^{ad}_{S \sqcup \ul{r}}$), modulo relabelling the elements $\{1,2,\ldots,r\}$, i.e.\ the induced map
\begin{equation}\label{eq:XXX}
\bigoplus_{r \geq 0} k\{F_r/\Sigma_r\} \lra H^0(\GL ; I^{\otimes S} \otimes (I^\vee)^{\otimes S} \otimes X^{\bullet})
\end{equation}
is an isomorphism for $\bullet < 2p$.

We wish to define an isomorphism
\begin{equation*}
\Upsilon: k\{d_0, d_1, d_2, \ldots\}^{\otimes S} \otimes k\{\Sigma_S^{ad}\} \overset{\sim}\lra \bigoplus_{r \geq 0} k\{F_r/\Sigma_r\}
\end{equation*}
which intertwines the maps \eqref{eq:OtherForm} (when $T=S$) and \eqref{eq:XXX}; as  \eqref{eq:XXX} is an isomorphism for $\bullet < 2p$, it will then follow that \eqref{eq:OtherForm} is also an isomorphism in this range. For a function $\ell : S \to \bN$ and a permutation $\sigma \in \Sigma_S^{ad}$ written in cycle form as
$$\sigma = (s_1, s_2, \ldots, s_{c_1})(s_{c_1+1}, s_{c_1+2}, \ldots, s_{c_1+c_2}) \cdots(s_{c_1 + \cdots + c_{i-1}+1}, \ldots, s_{c_1 + \cdots +c_i}),$$
we write $r := \sum_{s \in S} \ell(s)$ and define $\Upsilon(\left(\bigotimes_{s \in S} d_{\ell(s)}\right) \otimes \sigma)$ to be 
\begin{align*}
&( \underbrace{*, \ldots, *}_{\ell(s_1)}, s_1, \underbrace{*, \ldots, *}_{\ell(s_2)}, s_2, \cdots, \underbrace{*, \ldots, *}_{\ell(s_{c_1})}, s_{c_1})\\
& ( \underbrace{*, \ldots, *}_{\ell(s_{c_1+1})}, s_{c_1+1}, \underbrace{*, \ldots, *}_{\ell(s_{c_1+2})}, s_{c_1+2}, \cdots, \underbrace{*, \ldots, *}_{\ell(s_{c_1+c_2})}, s_{c_1+c_2}) \cdots \\
& (\,\,\,\,\,\,\,\,\,\, \mathclap{\underbrace{*, \ldots, *}_{\ell(s_{c_1 + \cdots + c_{i-1}+1})}}\,\,\,\,\,\,\,\,\,\, , s_{c_1 + \cdots + c_{i-1}+1}, \,\,\,\,\,\,\,\,\,\,\mathclap{\underbrace{*, \ldots, *}_{\ell(s_{c+1 + \cdots + c_{i-1}+2})}}\,\,\,\,\,\,\,\,\,\,, s_{c+1 + \cdots + c_{i-1}+2} , \cdots, \,\,\,\,\,\,\,\,\, \mathclap{\underbrace{*, \ldots, *}_{\ell(s_{c_1+\cdots+c_i})}}\,\,\,\,\,\,\,\,\,\, , s_{c_1+\cdots+c_i}) \in F_r/\Sigma_r
\end{align*}
where the $*$ denote the elements of $\ul{r}$. The map $\Upsilon$ obtained by extending this linearly is visibly an isomorphism: it even gives a bijection between the natural bases.

To verify that $\Upsilon$ intertwines the maps \eqref{eq:OtherForm} (when $T=S$) and \eqref{eq:XXX} we first observe that the target of these maps has a right $\Sigma_S$-action by permuting the $I^\vee$'s, and that these maps are $\Sigma_S$-equivariant if
\begin{enumerate}[(i)]
\item $k\{d_0, d_1, d_2, \ldots\}^{\otimes S} \otimes k\{\Sigma_S^{ad}\}$ is endowed with the right $\Sigma_S$-action by precomposition on $\Sigma_S^{ad}$ (and nothing on the first tensor factor), and

\item $F_r/\Sigma_r$ is endowed with the right $\Sigma_S$-action induced by precomposition on $F_r \subset \Sigma^{ad}_{S \sqcup \ul{r}}$, which does indeed preserve the subset $F_r$.
\end{enumerate}
The formula for $\Upsilon(\left(\bigotimes_{s \in S} d_{\ell(s)}\right) \otimes \sigma)$ above is the result of precomposing the element $\Upsilon(\left(\bigotimes_{s \in S} d_{\ell(s)}\right) \otimes \mathrm{Id}_S)$, i.e.\
$$(\underbrace{*, \ldots, *}_{\ell(s_1)}, s_1)(\underbrace{*, \ldots, *}_{\ell(s_2)}, s_2) \cdots (\underbrace{*, \ldots, *}_{\ell(s_{c_1})}, s_{c_1}) \cdots (\,\,\,\,\,\,\,\,\,\, \mathclap{\underbrace{*, \ldots, *}_{\ell(s_{c_1+\cdots+c_i})}}\,\,\,\,\,\,\,\,\,\, , s_{c_1+\cdots+c_i}) \in F_r/\Sigma_r,$$
with $\sigma$, so $\Upsilon$ is also $\Sigma_S$-equivariant for these actions. To check that $\Upsilon$ intertwines \eqref{eq:OtherForm} and \eqref{eq:XXX} it therefore suffices to shows that it does so on elements of the form $\left(\bigotimes_{s \in S} d_{\ell(s)}\right) \otimes \mathrm{Id}_S$. As the image of this element under \eqref{eq:OtherForm} is the cup product of the $d_{\ell(s)}$'s, by taking cup products it suffices to show that class $d_r$ is the image under \eqref{eq:XXX} of the $(r+1)$-cycle $(*, \ldots, *, \cdot) \in F_r/\Sigma_r \subset \Sigma^{ad}_{\{\cdot\} \sqcup \ul{r}}/\Sigma_r$, which it is by definition.
\end{proof}

\begin{corollary}\label{cor:IdOnBrauerCat}
For $i < 2p$ there is an identification
$$H^{i}(\GL ; I^{\otimes S} \otimes (I^\vee)^{\otimes T})\cong H^0(\GL ; I^{\otimes S} \otimes (I^\vee)^{\otimes T} \otimes X^i)$$
of functors from the upward walled Brauer category to $k$-modules.
\end{corollary}
\begin{proof}
Identify the domain of \eqref{eq:OtherForm} with the domain of \eqref{eq:FF} via $d_i \mapsto x^{[i]}$.
\end{proof}

\section{Proof of Theorem \ref{thm:mainFin}: the general case}\label{sec:GeneralCase}

\subsection{Some semisimplicity}\label{sec:Semisimplicity}

As long as $p > n,m$ the algebra $k[\Sigma_n \times \Sigma_m]$ is semisimple, and for partitions $\lambda \vdash n$ and $\mu \vdash m$ the external tensor products of Specht modules $S^\lambda \otimes S^\mu$ form a complete set of simple modules. Now $\Sigma_n \times \Sigma_m$ acts on $V_{[n,m]}$, and so defining
$$S_{\lambda, \mu}(V) := \Hom_{k[\Sigma_n \times \Sigma_m]} ( S^\lambda \otimes S^\mu, V_{[n,m]})$$
the evaluation map
$$\bigoplus_{\lambda, \mu} S^\lambda \otimes S^\mu \otimes S_{\lambda, \mu}(V) \lra V_{[n,m]}$$
is an isomorphism.

As a final ingredient we should like to know that the quotient map
\begin{equation}\label{eq:quot}
q: V^{\otimes n} \otimes (V^\vee)^{\otimes m} \lra V_{[n,m]}
\end{equation}
is split as a map of $\GL(V)$-modules. Unfortunately this is not generally true: in the exact sequence
\begin{equation}\label{eq:NonSplit}
0 \lra k \overset{coev}\lra V \otimes V^\vee \overset{q}\lra V_{[1,1]} \lra 0
\end{equation}
we have $ev : H_0(GL(V) ; V \otimes V^\vee) \overset{\sim}\to k$, and $ev \circ coev : k \to k$ is multiplication by $\dim(V)$; thus there is an exact sequence
$$\cdots \lra H_1(\GL(V) ; V_{[1,1]}) \overset{\partial}\lra k \overset{-\cdot\dim(V)}\lra k \overset{q_*}\lra H_0(\GL(V) ; V_{[1,1]}) \lra 0$$
and so if $\dim(V) \equiv 0 \mod p$ then $\partial$ is nontrivial and so \eqref{eq:NonSplit} cannot be $\GL(V)$-equivariantly split. However, we have the following partial result, which will suffice.

\begin{proposition}\label{prop:QuotSplit}
If $\dim(V) \equiv \tfrac{p+1}{2} \mod p$ and $n+m \leq \tfrac{p+1}{2}$ then the quotient map \eqref{eq:quot} is split as a map of $\GL(V)$-modules.
\end{proposition}
\begin{proof}
Let $\delta := \dim(V)$. In this proof we will use a result from the literature concerning the walled Brauer algebra $\mathcal{B}_{n,m}(\delta)$, i.e.\ the endomorphism algebra of the object $(\ul{n}, \ul{m})$ in the $k$-linear walled Brauer category $\mathsf{wBr}_\delta$ with parameter $\delta \in k$, which we described in Section \ref{sec:FullWBr}. Recall that there are commuting actions of $\mathcal{B}_{n,m}(\delta)$ and $k[\GL(V)]$ on $V^{\otimes n} \otimes (V^\vee)^{\otimes m}$.

Let $I \subset \mathcal{B}_{n,m}(\delta)$ denote the right sub-$\mathcal{B}_{n,m}(\delta)$-module spanned, as a $k$-module, by those walled Brauer diagrams which contain at least one matched pair, i.e.\ contain an arc which crosses the wall. This is a right sub-$\mathcal{B}_{n,m}(\delta)$-module as such diagrams must contain a wall-crossing arc at the left-hand end, and this is clearly preserved by right-multiplication. The exact sequence $0 \to I \to \mathcal{B}_{n,m}(\delta) \to \mathcal{B}_{n,m}(\delta)/I \to 0$ yields an exact sequence
$$\begin{tikzcd}[column sep=4pt]
   I \otimes_{\mathcal{B}_{n,m}(\delta)}(V^{\otimes n} \otimes (V^\vee)^{\otimes m}) \rar
             \ar[draw=none]{d}[name=X, anchor=center]{}
    & \mathcal{B}_{n,m}(\delta) \otimes_{\mathcal{B}_{n,m}(\delta)}(V^{\otimes n} \otimes (V^\vee)^{\otimes m}) \ar[rounded corners,
            to path={ -- ([xshift=2ex]\tikztostart.east)
                      |- (X.center) \tikztonodes
                      -| ([xshift=-2ex]\tikztotarget.west)
                      -- (\tikztotarget)}]{dl}[at end]{} \\      
   (\mathcal{B}_{n,m}(\delta)/I) \otimes_{\mathcal{B}_{n,m}(\delta)}(V^{\otimes n} \otimes (V^\vee)^{\otimes m}) \rar & 0
\end{tikzcd}$$
and the definition of $I$ identifies the quotient map with \eqref{eq:quot}. Thus in order to show that \eqref{eq:quot} is $\GL(V)$-equivariantly split, it suffices to show that $\mathcal{B}_{n,m}(\delta) \to \mathcal{B}_{n,m}(\delta)/I$ is split as a map of right $\mathcal{B}_{n,m}(\delta)$-modules.

For the latter is suffices that the algebra $\mathcal{B}_{n,m}(\delta)$ be semisimple, and for this we apply a result of Andersen--Stroppel--Tubbenhauer \cite{AST}. In the notation of that paper (Conventions 3.11 of loc.\ cit.), as $\dim(V) \equiv \tfrac{p+1}{2} \mod p$ (and $p$ is odd) we have $\delta_p = \tfrac{p+1}{2}$, so $\delta_p \neq 0$ and hence \cite[Theorem 6.1 (1)]{AST} applies and says that $\mathcal{B}_{n,m}(\delta)$ is semisimple providing $m+n \leq \min\{\delta_p+1, p-\delta_p+1\} = \tfrac{p+1}{2}$. (Note that our assumption $\dim(V) \equiv \tfrac{p+1}{2} \mod p$ was made to maximise the value of $m+n$ for which their result applies.)
\end{proof}

\subsection{Proof of Theorem \ref{thm:mainFin}}\label{sec:ProofGeneralCase}

To prove Theorem \ref{thm:mainFin}, using Corollary \ref{cor:GLvSL} it suffices to identify
\begin{equation}\label{eq:Reduced}
H^i(\GL(V) ; V_{[n,m]}) \cong [V_{[n,m]} \otimes X^i]^{\GL(V)}
\end{equation}
as $k[\Sigma_n \times \Sigma_m]$-modules, for $i < 2p$ and $n+m \leq \tfrac{p+1}{2}$ and all large enough $\dim(V)$. Applying the exact functor $\Hom_{k[\Sigma_n \times \Sigma_m]}(S^\lambda \otimes S^\mu, -)$ then gives the claimed result.

By the results of Section \ref{sec:Stability} the two sides stabilise with $\dim(V)$, so it suffices to establish this identity for all large enough $\dim(V)$ with $\dim(V) \equiv \tfrac{p+1}{2} \mod p$. In this case by Proposition \ref{prop:QuotSplit} and the assumption $n+m \leq \tfrac{p+1}{2}$ the quotient map in the sequence
$$\bigoplus_{i=1}^n \bigoplus_{j=1}^m V^{\otimes n-1} \otimes (V^\vee)^{\otimes m-1} \to V^{\otimes n} \otimes (V^\vee)^{\otimes m} \lra V_{[n,m]} \lra 0$$
is split as $\GL(V)$-modules and hence this sequence remains exact after applying $H^i(\GL(V) ; -)$ or $[- \otimes X^i]^{\GL(V)}$. Combining this with the natural isomorphism $H^i(\GL(V) ; -) \cong [- \otimes X^i]^{\GL(V)}$ (for $i < 2p$) of functors on the upward Brauer category, given by Corollary \ref{cor:IdOnBrauerCat}, yields the isomorphism \eqref{eq:Reduced} as required.

\begin{remark}[Stability range]\label{rem:StabRangeThmC}
As explained in Section \ref{sec:stab}, $H^i(\SL(V) ; V_{[n,m]})$ is in the stable range as long as $2i \leq \dim(V)-2-(n+m)$. One can verify in a similar way that $[V_{[n,m]} \otimes X^i]^{\SL(V)} = H^0(\SL(V) ; V_{[n,m]} \otimes X^i)$ is in the stable range when $i < 2p$ and $0 \leq \dim(V) -2-(n+m+i)$. Thus Theorem \ref{thm:mainFin} holds as long as $i < 2p$ and $\dim(V) \geq 2i+2+n+m$.
\end{remark}

\part{Stable cohomology of congruence subgroups}

For the rest of the paper all cohomology will be taken with $\bF_p$-coefficients, which we will often omit from the notation.

\section{Recollections}\label{sec:Recollections}

\subsection{Completed cohomology}
Following Calegari and Emerton \cite{CEStability} we consider the system of congruence subgroups defined as the kernels
\begin{equation}\label{eq:Extension}
1 \lra \SL_n(\bZ, p^r) \lra \SL_n(\bZ) \lra \SL_n(\bZ/p^r) \lra 1
\end{equation}
as a pro-group, $\{\SL_n(\bZ, p^r)\}_r$, and let the associated completed ($\bF_p$-)cohomology\footnote{We use the notation $\widetilde{H}$ for completed cohomology, following Calegari and Emerton, and pre-emptively apologise for the confusion with reduced cohomology that it will no doubt cause topologists.}
$$\widetilde{H}^i(\SL_n(\bZ)) := \colim_{r \to \infty} H^i(\SL_n(\bZ, p^r) )$$
be the cohomology of this pro-group. It of course depends on the (implicit) prime $p$. 

We also consider the pro-(finite group) $\SL_n(\bZ_p) := \{\SL_n(\bZ/p^r)\}_r$, whose cohomology we call the continuous cohomology of $\SL_n(\bZ_p)$. Taking the limit of this pro-(finite group) gives the usual group $\SL_n(\bZ_p)$ with the topology induced from the $p$-adic topology on $\bZ_p$, and the pro-group plays for us the role of this topological group: working with the pro-object allows us to form the ``classifying pro-space'' $\{B\SL_n(\bZ/p^r)\}_r$, which should be thought of as an implementation of the classifying space of the group $\SL_n(\bZ_p)$ which takes into account the topology.

By the main theorem of \cite{CEStability} the completed cohomology groups enjoy homological stability with respect to $n$, and (hence) the outer action of the pro-(finite group) $ \{\SL_n(\bZ/p^r)\}_r$ on these groups is trivial in the stable range. This also follows from slightly more recent work of Iwasa \cite[Theorem 1.3]{Iwasa}. We write
$$\widetilde{H}^i(\SL(\bZ)) := \lim_{n \to \infty} \widetilde{H}^i(\SL_n(\bZ))$$
for the stable completed cohomology. We may identify it with the cohomology of the pro-space
$$\hocolim_{n \to \infty} \, \{B\SL_n(\bZ, p^r)\}_r,$$
the homotopy colimit formed in pro-spaces. Generally speaking (homotopy) limits and colimits in pro-spaces are not formed object-wise, and so this should not be confused with the pro-space $\{\hocolim_{n \to \infty} B\SL_n(\bZ, p^r)\}_r$, which will play no role. For background on the homotopy theory of pro-spaces see \cite[Section 2 and Appendices]{ArtinMazur} and \cite{IsaksenPro}.

Taking colimits of the Leray--Hochschild--Serre spectral sequences of the extensions \eqref{eq:Extension} gives a spectral sequence
$$E_2^{s,t}  = H^s_\mathrm{cts}(\SL_n(\bZ_p) ; \widetilde{H}^t(\SL_n(\bZ))) \Longrightarrow H^{s+t}(\SL_n(\bZ)),$$
and in the stable range the coefficient system is untwisted. Here the continuous cohomology refers to 
$$H^s_\mathrm{cts}(\SL_n(\bZ_p) ; M) := \colim_{r \to \infty} H^s(\SL_n(\bZ/p^r) ; M)$$
and is defined for $\SL_n(\bZ_p)$-modules $M$ for which the action factors through $\SL_n(\bZ_p) \to \SL_n(\bZ/p^m)$ for some $m$ (which allows us to form the colimit).

More generally, considering the extension of pro-groups
$$1 \lra \{\SL_n(\bZ, p^{k+r})\}_r \lra \SL_n(\bZ, p^k) \lra \{\SL_n(\bZ/p^{r+k}, p^k)\}_r \lra 1$$
gives a spectral sequence
\begin{equation}\label{eq:CongruenceSS}
E_2^{s,t}  = H^s_\mathrm{cts}(\SL_n(\bZ_p, p^k) ; \widetilde{H}^t(\SL_n(\bZ) )) \Longrightarrow H^{s+t}(\SL_n(\bZ, p^k)),
\end{equation}
again untwisted in a stable range, whose $E_2$-term is quite accessible for $k > 0$ as then $\SL_n(\bZ_p, p^k)$ is a $p$-adic analytic group (see Section \ref{sec:AnalyticGps}).

\subsection{Relation to algebraic $K$-theory}\label{sec:RelKThy}
We rephrase \cite[Section 2.3]{Calegari}. Let $\K(-)$ denote the algebraic $K$-theory spectrum, so that there are acyclic maps $B\GL(-) \to \Omega_0^\infty \K(-)$. Write $\K^\mathrm{cts}(\bZ_p) := \holim\limits_{r \to \infty} \K(\bZ/p^r)$ for the continuous $K$-theory of $\bZ_p$, so there are induced maps of pro-spectra
$$\K(\bZ) \lra \K(\bZ_p) \lra \K^\mathrm{cts}(\bZ_p) \lra \{\K(\bZ/p^r)\}_r;$$
here and below we implicitly identify objects with constant pro-objects. The middle map is comparing the $K$-theory of $\bZ_p$ considered as a discrete ring with its continuous $K$-theory. Denote the 1-connected cover\footnote{We use the standard notation $\tau_{> d} X$ for the $d$-connected cover of a space or spectrum $X$, and $\tau_{\leq d} X$ for its $d$-th truncation. We allow ourselves to use both strict and non-strict inequalities in this style of notation, and also use $\tau_{[a,b]} X = \tau_{\leq b} \tau_{\geq a} X$.} of the algebraic $K$-theory spectrum by $\SK(-) = \tau_{>1}\K(-)$, and write
$$\kappa : \SK(\bZ) \lra \SK(\bZ_p)$$
for the map induced by $\bZ \to \bZ_p$, which we call the \emph{completion map}. Using this we form the diagram of pro-spaces
\begin{equation*}
\begin{tikzcd}[column sep=small]
\hocolim\limits_{n \to \infty}\{B\SL_n(\bZ, p^r)\}_r \dar \rar& \{\mathrm{hofib}(\Omega^\infty\kappa''_r)\}_r \dar & \mathrm{hofib}(\Omega^\infty\kappa') \dar \lar & \mathrm{hofib}(\Omega^\infty \kappa) \dar \lar\\
\hocolim\limits_{n \to \infty} B\SL_n(\bZ) \dar \rar& \Omega^\infty \SK(\bZ) \arrow[r, equals] \dar{\{\Omega^\infty\kappa''_r\}_r} & \Omega^\infty \SK(\bZ) \dar{\Omega^\infty\kappa'} \arrow[r, equals]& \Omega^\infty \SK(\bZ) \dar{\Omega^\infty \kappa}\\
\hocolim\limits_{n \to \infty} \{B\SL_n(\bZ/ p^r)\}_r \rar& \{\Omega^\infty \SK(\bZ/p^r)\}_r  & \Omega^\infty \SK^\mathrm{cts}(\bZ_p) \lar & \Omega^\infty \SK(\bZ_p) \lar
\end{tikzcd}
\end{equation*}
whose columns are fibration sequences. (Fibres, and more generally finite limits, of pro-spaces may be computed object-wise, see \cite[Appendix (4.1)]{ArtinMazur}.) All columns but the first are fibrations of pro-(infinite loop spaces), so the coefficient systems given by the cohomology of the fibres is trivial. The same property holds for the first column by the main theorem of \cite{CEStability}.

\begin{lemma}\label{lem:HorizIso}
All the horizontal maps in this diagram induce isomorphisms on $\bF_p$-cohomology.
\end{lemma}
\begin{proof}
The map in the middle row is a cohomology isomorphism, because it arises as a cover of the acyclic map $B\GL(\bZ) \to \Omega_0^\infty \K(\bZ)$.

For the left-hand map of the bottom row, the natural map of pro-spaces
$$\hocolim\limits_{n \to \infty} \{B\SL_n(\bZ/ p^r)\}_r \lra  \{ \hocolim\limits_{n \to \infty} B\SL_n(\bZ/ p^r)\}_r \simeq \{  B\SL(\bZ/ p^r)\}_r$$
is a cohomology isomorphism, because for each $r$ the sequence $B\SL_n(\bZ/ p^r)$ enjoys homological stability with respect to $n$, with a stability range that is independent of $r$ (this follows from \cite[Theorem 4.11]{vanderkallen}). Furthermore the maps $B\SL(\bZ/ p^r) \to \Omega^\infty \SK(\bZ/p^r)$ are cohomology isomorphisms, as in the first paragraph of this proof.

For the middle map of the bottom row, the argument of \cite[Sublemma 2.18]{Calegari} shows that for each $i$ the dimension of $H_i(\Omega^\infty \SK(\bZ/p^r) )$ is finite and bounded independently of $r$. By \cite[Theorem B]{GoerssLimit}, using that the homology of infinite loop spaces are abelian Hopf algebras, it follows that the map
$$H_*(\Omega^\infty \SK^\mathrm{cts}(\bZ_p)) \lra \lim_{r \to \infty} H_*(\Omega^\infty \SK(\bZ/p^r) )$$
is an isomorphism (in principle the limit is taken in the category of $\bF_p$-coalgebras, but the uniform boundedness means that the limit in $\bF_p$-modules has finite type, and so inherits a coalgebra structure and agrees with the limit in $\bF_p$-coalgebras). Dualising shows that this map is a $\bF_p$-cohomology isomorphism.

The right-hand map of the bottom row is a $\bF_p$-cohomology isomorphism by \cite[Theorem C (iii)]{HM1} applied with $k=\bF_p$ and $A = W(k) = \bZ_p$. That theorem is formulated as an equivalence of $p$-adically complete spectra, but this yields an equivalence of their $p$-adically complete infinite loop spaces, and therefore an $\bF_p$-(co)homology equivalence between their infinite loop spaces.

For each of the maps in the top row of the diagram, we use that for each column the action of the fundamental group of the base on the cohomology of the fibre is trivial, then apply the Zeeman comparison theorem to each of the maps of Serre spectral sequences between the columns.
\end{proof}

As $\widetilde{H}^*(\SL(\bZ))$ is the cohomology of the pro-space $\hocolim_{n \to \infty}\{B\SL_n(\bZ, p^r)\}_r$, the top row of the diagram provides an isomorphism
\begin{equation}\label{eq:FibKappaIsCompletedCoh}
\widetilde{H}^*(\SL(\bZ)) \cong H^*(\mathrm{hofib}(\Omega^\infty\kappa)),
\end{equation}
relating the stable completed cohomology to the cohomology of the fibre of the completion map in $K$-theory. In the following two sections we will explain how the latter may be determined, at least in degrees $* < 2p-2$, using deep results in algebraic $K$-theory. We denote by $\K(-;\bZ_p) := \holim_{m \to \infty} \K(-)/p^m$ the $p$-adic completion of the algebraic $K$-theory spectrum $\K(-)$, which we shall use to formulate things.

\subsection{The completion map}\label{sec:CompletionMap}

In this section we will describe the structure of the completion map
$$\kappa : \SK(\bZ ; \bZ_p) \lra \SK(\bZ_p ; \bZ_p)$$
in stable homotopy theoretic terms, for $p$ an odd prime. We first recall some well-known spectra and their homotopy groups: a source for this is \cite[pp.\ 204-206]{AdamsBook} and \cite[Section 5]{AdamsVect}. We will make use of the connective real and complex topological $K$-theory spectra $ko$ and $ku$, and their periodic versions $KO$ and $KU$, all of which we implicitly complete at $p$. We will also make use of the (connective, $p$-complete) image-of-$J$ spectrum $j$, defined by the fibre sequence
$$j \lra ku \overset{\psi^r-1}\lra \tau_{>0} ku$$
for $r$ an integer which topologically generates $\bZ_p^\times$ and $\psi^r$ the corresponding Adams operation. 

The homotopy groups of these spectra are as follows. We have $\pi_*(ku) = \bZ_p[u]$ for $u \in \pi_2(ku)$ the Bott element, and $\pi_*(ko) = \bZ_p[u^2]$ considered as a subring of $\pi_*(ku)$ via complexification $c : ko \to ku$. The Adams operation $\psi^r$ acts on the Bott class $u \in \pi_2(ku)$ by multiplication by $r$, so the map $\psi^r-1 : ku \to \tau_{>0}ku$ acts on homotopy groups as $u^{k} \mapsto(r^{k}-1) \cdot u^{k}$. As $r \in \bZ_p^\times$ is a topological generator, $(r^{k}-1)$ is a unit in $\bZ_p$ if $p-1 \nmid k$ and is $p \cdot i$ times a unit if $k= (p-1) \cdot i$, giving
$$\pi_*(j) = \begin{cases}
\bZ_p & *=0\\
\bZ_p/p \cdot i & * = i(2p-2)-1\\
0 & \text{else}.
\end{cases}$$

\subsubsection{$K$-theory of $\bZ_p$} By the main theorem of \cite{BokstedtMadsen} there is an equivalence 
\begin{equation}\label{eq:BokstedtMadsen}
\K(\bZ_p ; \bZ_p) \simeq j \oplus \Sigma j \oplus \Sigma^3 ku.
\end{equation}
At the level of homotopy groups this gives
$$SK_*(\bZ_p ; \bZ_p) =  \begin{cases}
\bZ/p & *=2p-3\\
0 & \text{else}
\end{cases} \oplus  \begin{cases}
\bZ_p & *=3,5,7,9,\ldots\\
0 & \text{else}
\end{cases} \text{ for } * < 2p-2.$$
Note that the contributions of $\pi_0(j)=\bZ_p$ and $\pi_1(\Sigma j) = \bZ_p$ to \eqref{eq:BokstedtMadsen} are removed by passing to the 1-connected cover $\SK(\bZ_p;\bZ_p)$.

\subsubsection{$K$-theory of $\bZ$ and the completion map} The structure of $K_*(\bZ;\bZ_p)$ (at irregular primes) is far more complicated. Following Weibel's survey \cite{WeibelIntegers}, we have
$$SK_*(\bZ ; \bZ_p) \cong \begin{cases}
\text{order } |\bZ_p/\mathrm{Num}(B_{2k}/4k)| & *=4k-2\\
0 & *=4k-1\\
? & *=4k\\
\bZ_p & * = 4k+1\\
\end{cases}  \text{ for } * < 2(2p-3)$$
where $B_{2k}$ denotes the $2k$th Bernoulli number, and $?$ is unknown but finite. If the Vandiver conjecture holds for $p$, then in degrees $\equiv 0 \mod 4$ these groups vanish and in degrees $\equiv 2\mod 4$ they are cyclic. The Vandiver conjecture has been checked for $p \leq 2^{31}$, cf.\ \cite{HHO}. 

This can be extracted from \cite{WeibelIntegers} as follows: In degrees $\equiv 1,3 \mod 4$ it follows from Theorem 1 and the description of $w_i(\bQ)$ in Lemma 27; in degrees $\equiv 2 \mod 4$ it follows from Corollary 95; in degrees $\equiv 0 \mod 4$ it follows from Theorem 6; the claim involving the Vandiver conjecture follows from Corollary 107.

We wish to show that the truncations 
$$\tau_{[2, 2p-3]} \K(\bZ; \bZ_p) = \tau_{\leq 2p-3} \tau_{\geq 2} \K(\bZ; \bZ_p) = \tau_{\leq 2p-3}\SK(\bZ; \bZ_p)$$
and $\tau_{[2, 2p-3]} \K(\bZ_p; \bZ_p)$ are coproducts of Eilenberg--MacLane spectra, and thereby understand to some extent the map induced by $\kappa$ between these truncations by understanding its effect on homotopy groups. We will use the following lemma to control stable homotopy classes of maps between Eilenberg--MacLane spectra. For $\bZ$-modules $A$ and $B$, the Universal Coefficient Theorem identifies $[HA, HB]$ with $\Hom_{\bZ}(A,B)$, and $[HA, \Sigma HB]$ with $\Ext^1_{\bZ}(A, B)$, functorially in $A$ and $B$. 

\begin{lemma}
If $A$ and $B$ are finitely-generated $\bZ_p$-modules, then $[HA, \Sigma^i HB]=0$ for $1 < i < 2p-2$.
\end{lemma}
\begin{proof}
Recall that $[H\bZ/p, \Sigma^*H\bZ/p]$ is by definition the $\bZ/p$-Steenrod algebra $\mathcal{A}_p$, and (for $p$ odd) this is generated under composition by the Bockstein map $\beta : H\bZ/p \to \Sigma H\bZ/p$ as well as operations $\mathcal{P}^n : H\bZ/p \to \Sigma^{2n(p-1)} H\bZ/p$ for $n \geq 1$, subject to certain relations which need not concern us. Recall that there is a cofibre sequence $H\bZ_p \overset{p}\to H\bZ_p \overset{\rho}\to H\bZ/p \overset{\partial}\to \Sigma H\bZ_p$ and the Bockstein is defined to be $\beta := (\Sigma\rho) \circ \partial$. There is a commutative diagram
\begin{equation*}
\begin{tikzcd}[column sep=10pt]
& & & {[ H\bZ/p, \Sigma^{*-1} H\bZ/p]} \dar{- \circ \rho} \arrow[ld, swap, "- \circ \beta"]\\
{[H\bZ_p, \Sigma^* H\bZ/p]} &  {[H\bZ_p, \Sigma^* H\bZ/p]} \arrow[l, swap, "0"] & {[H\bZ/p, \Sigma^* H\bZ/p]} \arrow[l, swap, "- \circ \rho"] & {[ H\bZ_p, \Sigma^{*-1} H\bZ/p]} \arrow[l, swap, "- \circ \partial"]
\end{tikzcd}
\end{equation*}
where the bottom row is exact and is given by applying $[-, \Sigma^* H\bZ/p]$ to this cofibre sequence, and the vertical map is a (degree-shifted) copy of the horizontal map $- \circ \rho$. The bottom row shows that $- \circ \rho$ is surjective, so it follows that the image of $- \circ \partial$ in $\mathcal{A}_p = {[H\bZ/p, \Sigma^* H\bZ/p]}$ is precisely the left $\mathcal{A}_p$-submodule generated by $\beta$, and therefore identifies $[H\bZ_p, \Sigma^* H\bZ/p] \cong  \mathcal{A}_p / \mathcal{A}_p \beta$. By the description of the Steenrod algebra this vanishes in degrees $0 < * < 2p-2$. 

Any finitely-generated $\bZ_p$-module is a finite sum of $\bZ_p$'s and $\bZ/p^r$'s. Applying $[-, \Sigma^*H\bZ/p]$ to the cofibre sequence $H\bZ_p \overset{p^r}\to H\bZ_p \to H\bZ/p^r \to \Sigma H\bZ_p$ gives a short exact sequence
$$0 \longleftarrow  {[H\bZ_p, \Sigma^* H\bZ/p]}  \longleftarrow {[H\bZ/p^r, \Sigma^* H\bZ/p]} \longleftarrow {[ H\bZ_p, \Sigma^{*-1} H\bZ/p]} \longleftarrow 0 $$
so using $[H\bZ_p, \Sigma^* H\bZ/p] \cong  \mathcal{A}_p / \mathcal{A}_p \beta$, which vanishes in degrees $0 < * < 2p-2$, we see that $[H\bZ/p^r, \Sigma^* H\bZ/p]=0$ for $1 < * < 2p-2$. This shows that the claim holds for $B=\bZ/p$ and for all finitely-generated $\bZ_p$-modules $A$. Applying $[HA, \Sigma^*-]$ to the cofibre sequences $H\bZ/p^{r-1} \to H\bZ/p^r \to H\bZ/p \to \Sigma H\bZ/p^{r-1}$, then shows---by induction on $r$---that the claim holds for $B=\bZ/p^r$ and for all finitely-generated $\bZ_p$-modules $A$. 

Finally, writing $H\bZ_p = \holim\limits_r H\bZ/p^r$ and applying $[HA, \Sigma^*-]$ to it gives a Milnor sequence
$$0 \lra \lim_{r \to \infty}{}^1 [HA, \Sigma^{i-1} H\bZ/p^r] \lra [HA, \Sigma^i H\bZ_p] \lra \lim_{r \to \infty} [HA, \Sigma^i H\bZ/p^r] \lra 0.$$
The outer terms vanish for $i > 2$. For $i=2$ the right-hand term vanishes, and the left-hand term is $\lim{_{r \to \infty}^1} \Ext^1_{\bZ}(A, \bZ/p^r)$. As $\mathrm{Ext}^2_\bZ(-,-)$ vanishes identically, $\Ext^1_{\bZ}(A, -)$ preserves epimorphisms and so $\{\Ext^1_{\bZ}(A, \bZ/p^r)\}_r$ is an inverse system of epimorphisms and hence has vanishing $\lim^1$ by the Mittag-Leffler condition.
\end{proof}

We use this lemma to show that $\tau_{[2, 2p-3]} \K(\bZ; \bZ_p)$ and $\tau_{[2, 2p-3]} \K(\bZ_p; \bZ_p)$ are coproducts of Eilenberg--MacLane spectra, by showing that their Postnikov towers must split: assuming a splitting $\tau_{[2, i]} \K(\bZ; \bZ_p) \simeq \bigoplus_{j=2}^i \Sigma^j HK_j(\bZ;\bZ_p)$ has been chosen, there is a pullback
\begin{equation*}
\begin{tikzcd}
\tau_{[2, i+1]} \K(\bZ; \bZ_p) \dar \rar & * \dar\\
\bigoplus_{j=2}^i \Sigma^j HK_j(\bZ;\bZ_p) \rar & \Sigma^{i+2} HK_{i+1}(\bZ;\bZ_p)
\end{tikzcd}
\end{equation*}
but the lower map is nullhomotopic as long as $i < 2p-2$ by the lemma. However, this splitting is not completely canonical: when $i=4k$ the nullhomotopy of the lower map may not be unique (though it is if the Vandiver conjecture holds). The analogous discussion goes through for $\tau_{[2, 2p-3]} \K(\bZ_p; \bZ_p)$, though in this case the splitting is canonical.

Using the lemma again it follows that the map $\kappa$ on truncations is determined by its components
\begin{align*}
\kappa_{4k-2} &: \Sigma^{4k-2} HK_{4k-2}(\bZ;\bZ_p) \lra \Sigma^{4k-1}HK_{4k-1}(\bZ_p;\bZ_p)\\
\kappa_{4k} &: \Sigma^{4k} HK_{4k}(\bZ;\bZ_p) \lra \Sigma^{4k+1}HK_{4k+1}(\bZ_p;\bZ_p)\\
\kappa_{4k+1} &: \Sigma^{4k+1} HK_{4k+1}(\bZ;\bZ_p)\lra \Sigma^{4k+1} HK_{4k+1}(\bZ_p;\bZ_p).
\end{align*}

\subsubsection{The completion map at regular odd primes}\label{sec:CompletionRegular} If $p$ is a regular odd prime then the completion map may be completely described at the level of spectra, improving upon the description in the previous section. Our reference for the following is \cite[Section 3]{RognesRegular}. It follows from the (affirmed) Quillen--Lichtenbaum conjecture that there is an equivalence $\K(\bZ; \bZ_p) \simeq j \oplus \Sigma^5 ko$; similarly, \eqref{eq:BokstedtMadsen} gives an equivalence $\K(\bZ_p ; \bZ_p) \simeq j \oplus \Sigma j \oplus \Sigma^3 ku$. 

\begin{lemma}[Rognes]
Under these equivalences the completion map is the identity on the $j$-summand and on the $\Sigma^5 ko$-summand is the map $\chi: \Sigma^5 ko \to\Sigma^3 ku$ induced by the suspension of the complexification map $\Sigma c : \Sigma ko \to \Sigma ku$ by taking 1-connected covers. 
\end{lemma}
\begin{proof}[Proof sketch]
Following \cite[Section 3]{RognesRegular} we develop the diagram
\begin{equation*}
\begin{tikzcd}[column sep=small]
\K(\bZ;\bZ_p) \rar \dar & \K(\bZ[\tfrac{1}{p}];\bZ_p) \rar{\sim} \dar& \tau_{\geq 0}\K^{\text{\'et}}(\bZ[\tfrac{1}{p}];\bZ_p) \rar{\sim} \dar& j \oplus \Sigma ko \dar{\mathrm{Id}_j \oplus 0 \oplus \Sigma c}\\
\K(\bZ_p;\bZ_p) \rar & \K(\bQ_p;\bZ_p) \rar{\sim} & \tau_{\geq 0}\K^{\text{\'et}}(\bQ_p;\bZ_p) \rar{\sim} & j \oplus \Sigma j \oplus  \Sigma ku
\end{tikzcd}
\end{equation*}
where the right-hand square is as indicated by \cite[Proposition 3.1]{RognesRegular}, the horizontal maps in the middle square are equivalences by the (affirmed) Quillen--Lichtenbaum conjecture in the two cases, and the left-hand square is cartesian by the localisation sequence in $K$-theory. More precisely, the map of horizontal fibres in the left-hand square is identified with $\K(\bZ/p ; \bZ_p) \overset{\sim}\to \K(\bZ_p/p ; \bZ_p) \simeq H\bZ_p$. In each row this copy of $H\bZ_p$ cancels against the lowest homotopy group of $\Sigma ko$ or $\Sigma ku$ (by considering the left-hand square on $\pi_1$), giving the claimed description of the completion map in terms of the lift of $\Sigma c : \Sigma ko \to \Sigma ku$ to 1-connected covers (see first line of the proof of \cite[Theorem 3.8]{RognesRegular}).
\end{proof}

Taking the suspension of the Wood cofibre sequence $\Sigma KO \overset{\eta}\to KO \overset{c}\to KU$ and then taking 1-connective covers gives a cofibre sequence $\Sigma^2 ko \to \Sigma^5 ko \overset{\chi}\to \Sigma^3 ku$. Taking 1-connected covers of the maps in the Lemma, we see that $\kappa$
is identified with $\mathrm{Id}_{\tau_{>1} j} \oplus \chi : \tau_{>1}j \oplus \Sigma^5 ko \to \tau_{>1}j \oplus \tau_{>1} \Sigma j \oplus \Sigma^3 ku$ giving an equivalence 
\begin{equation*}
\mathrm{hofib}(\kappa) \simeq \tau_{>0} j \oplus \Sigma^2 ko.
\end{equation*}

\subsection{Transgressive fibrations}\label{sec:Trans}

The discussion so far gives an analysis of the stable homotopy types of $\tau_{[2, 2p-3]} \K(\bZ; \bZ_p)$ and $\tau_{[2, 2p-3]} \K(\bZ_p; \bZ_p)$, and of the map between them: we now wish to take their associated infinite loop spaces in order to describe the behaviour of the Serre spectral sequence associated to the fibration
$$\Omega^\infty\kappa : \Omega^\infty \SK(\bZ; \bZ_p) \lra \Omega^\infty \SK(\bZ_p; \bZ_p)$$
in degrees $* < 2p-2$.

Let us say that a fibration $\pi: E \to B$ with 0-connected fibre $F$ is \emph{transgressive} (with $\bF_p$-coefficients) if $\pi_1(B)$ acts trivially on $H^*(F;\bF_p)$, and if $H^*(F;\bF_p)$ is freely generated as a graded-commutative $\bF_p$-algebra by a set of classes which are transgressive in the Serre spectral sequence for $\pi$. We say $\pi$ is \emph{trangressive in degrees $*<N$} if the above two conditions hold in this range of cohomological degrees. The class of such fibrations is closed under forming pullbacks, and (by the K{\"u}nneth theorem) is closed under forming products of fibrations. The following lemma gives a class of examples.

\begin{lemma}\label{lem:Transg}
The following fibrations are transgressive in degrees $* < 2p-2$:
\begin{enumerate}[(i)]
\item $\pi: E \to B$ a principal $K(A,n)$-fibration, for $A$ a finitely-generated $\bZ_p$-module.

\item $\alpha : K(A, n) \to K(B, n+1)$ for $A$ and $B$ finitely-generated $\bZ_p$-modules, and $\alpha \in H^{n+1}(K(A,n) ; B) \cong \Ext^1_{\bZ}(A, B) \cong [HA, \Sigma HB]$.

\item $N\cdot \mathrm{Id} + \alpha : K(\bZ_p, n+1) \times K(A, n) \to K(\bZ_p, n+1)$ for $N \in \bZ_p$, $A$ a  finitely-generated $\bZ_p$-module, and $\alpha \in H^{n+1}(K(A,n) ; \bZ_p) \cong \Ext^1_{\bZ}(A, \bZ_p) \cong [HA, \Sigma H\bZ_p]$. 
\end{enumerate}
\end{lemma}
\begin{proof}
For (i) note that the path fibration $PK(A, n+1) \to K(A,n+1)$ is transgressive in degrees $* < n+2p-2$ for $A$ either $\bZ_p$ or $\bZ/p^r$, by the calculation of the cohomology of Eilenberg--MacLane spaces in \cite{SC} (and the fact that $K(\bZ, n+1) \to K(\bZ_p, n+1)$ is a $\bF_p$-cohomology isomorphism). By taking products of fibrations the same holds when $A$ is a finitely-generated $\bZ_p$-module, and by taking pullbacks it holds for all principal $K(A,n)$-fibrations.

For (ii), this fibration arises from delooping the extension $0 \to B \to A' \to A \to 0$ classified by $\alpha$, so is a principal $K(A', n)$-fibration and so is transgressive by (i).

For (iii), if $N=0$ then this fibration is the product of the fibrations $K(\bZ_p, n+1) \to *$ and $\alpha : K(A, n) \to K(\bZ_p, n+1)$. The first is clearly transgressive, and the second is too by (ii). If $N \neq 0$ then the homotopy groups of the fibre $F$ of this fibration sit in an exact sequence
$$ \cdots 0 \lra \pi_{n+1}(F) \lra \bZ_p \overset{N}\lra \bZ_p \lra \pi_n(F) \lra A \lra 0 \lra \pi_{n-1}(F) \lra 0 \cdots$$
so $F \simeq K(A', n)$ for the extension $0 \to \bZ_p/N \to A' \to A \to 0$ determined by $\alpha$ and reduction modulo $N$. Furthermore, the fibration is principal as both maps $\mathrm{Id}$ and $\alpha$ deloop, so (i) applies.
\end{proof}

\begin{corollary}
The fibration 
$$\Omega^\infty\tau_{[2, 2p-3]}\kappa : \Omega^\infty\tau_{[2, 2p-3]} \K(\bZ; \bZ_p) \lra \Omega^\infty\tau_{[2, 2p-3]} \K(\bZ_p; \bZ_p)$$
 is transgressive in degrees $* < 2p-2$.
\end{corollary}
\begin{proof}
By the discussion in Section \ref{sec:CompletionMap} it suffices to show that each of the fibrations
\begin{align*}
 K(K_{4k-2}(\bZ;\bZ_p), 4k-2) &\xrightarrow{\Omega^\infty \kappa_{4k-2} } K(K_{4k-1}(\bZ_p;\bZ_p), 4k-1)\\
 K(K_{4k}(\bZ;\bZ_p), 4k) \times K(K_{4k+1}(\bZ;\bZ_p), 4k+1) &\xrightarrow{\Omega^\infty (\kappa_{4k} + \kappa_{4k+1})} K(K_{4k+1}(\bZ_p;\bZ_p), 4k+1)
\end{align*}
are transgressive. Using that $K_{4k- 1}(\bZ_p;\bZ_p)$ is $\bZ_p$ or $\bZ/p \oplus \bZ_p$, Lemma \ref{lem:Transg} (ii) applies to the first; using that $K_{4k+ 1}(\bZ_p;\bZ_p) \cong \bZ_p$, Lemma \ref{lem:Transg} (iii) applies to the second.
\end{proof}

Using Lemma \ref{lem:HorizIso} and the discussion surrounding it, this translates into the following statement.

\begin{corollary}\label{cor:TransSS}
For all large enough $n$, in degrees $* < 2p-2$ the group $\SL_n(\bZ_p)$ acts trivially on $\widetilde{H}^*(\SL_n(\bZ))$, and the latter is freely generated as a graded-commutative $\bF_p$-algebra by elements $\{x_\alpha\}_{\alpha \in I}$ which are transgressive in the spectral sequence
\begin{equation*}
E_2^{s,t}  = H^s_\mathrm{cts}(\SL_n(\bZ_p) ; \widetilde{H}^t(\SL_n(\bZ) )) \Longrightarrow H^{s+t}(\SL_n(\bZ)),
\end{equation*}
given by \eqref{eq:CongruenceSS} with $k=0$.\qed
\end{corollary}

\subsubsection{The spectral sequence at regular odd primes} When $p$ is a regular odd prime, the discussion in Section \ref{sec:CompletionRegular} gives an explicit description of this spectral sequence. Limiting our interest to the slightly smaller range of degrees $* < 2p-3$ (as we will later limit ourselves to $* < p-1$ anyway) and using that $\tau_{[1,2p-4]}j = 0$ by the description of the homotopy groups of $j$ at the beginning of Section \ref{sec:CompletionMap}, we see that in this range of degrees the right-hand column of in the large diagram in Section \ref{sec:RelKThy} may be identified with the fibre sequence
\begin{equation*}
\Omega^{\infty} \Sigma^2 ko \lra \Omega^\infty \Sigma^5 ko \lra \Omega^\infty \Sigma^3 ku.
\end{equation*}
The cohomology of these spaces is well-known, and can be extracted from \cite{DL} for example. To do so one should express these spaces in terms of their usual names from real and complex Bott periodicity: the fibre is $SO/U$, the base is the universal cover of $U$ (i.e. $SU$), and the total space is the universal cover of $U/Sp$. In degrees $* < 2p-3$ we obtain isomorphisms
\begin{align*}
\bF_p[x_2, x_6, x_{10}, \ldots] &\cong H^*(\mathrm{hofib}(\Omega^\infty\kappa)) \cong_{\text{\eqref{eq:FibKappaIsCompletedCoh}}} \widetilde{H}^*(\SL(\bZ))\\
\Lambda^*_{\bF_p}[ y_5, y_9, y_{13}, \ldots] &\cong H^*(\Omega^\infty \SK(\bZ))\\
\Lambda^*_{\bF_p}[ y_3, y_5, y_{7}, \ldots] &\cong H^*(\Omega^\infty \SK(\bZ_p)).
\end{align*}
The first row of isomorphisms calculates the completed cohomology $\widetilde{H}^*(\SL(\bZ))$ in degrees $* < 2p-3$, and so justifies the deduction of Corollary \ref{cor:mainReg} from Theorem \ref{thm:mainCong}. The remaining isomorphisms show, following \cite{DL}, that in this range the Serre spectral sequence for the right-hand column in the large diagram in Section \ref{sec:RelKThy} has the form
$$\Lambda^*_{\bF_p}[ y_3, y_5, y_7, \ldots] \otimes \bF_p[x_2, x_6, x_{10}, \ldots] \Longrightarrow \Lambda^*_{\bF_p}[ y_5, y_9, y_{13}, \ldots]$$
with generating differentials $d_{4k+3}(x_{4k+2}) = y_{4k+3}$, for an appropriate choice of generators $x_i$ and $y_i$. Using the large diagram in Section \ref{sec:RelKThy}, this describes the spectral sequence \eqref{eq:CongruenceSS} for $k=0$, in degrees $* < 2p-3$ and for all large enough $n$.

\subsection{Evaluating completed cohomology at irregular primes}\label{sec:LFunction}

To justify the examples from the introduction, we need to calculate $H^*(\mathrm{hofib}(\Omega^\infty \kappa))$, and in degrees $* \leq 2p-3$ this can be formally obtained from $\pi_*(\mathrm{hofib}(\kappa))$. The recent work of Blumberg--Mandell \cite{BlumbergMandell} expresses these homotopy groups in terms of {\'e}tale cohomology of $\bZ[\tfrac{1}{p}]$, but we shall proceed in a more down-to-earth way.

We are only interested in degrees $* < p-1$, and in this range the long exact sequence on homotopy groups for the map $\kappa$ has the form
$$
\begin{tikzcd}
0\rar &  \pi_{4k+1}(\mathrm{hofib}(\kappa)) \rar & \bZ_p \rar{\kappa_{4k+1}}
             \ar[draw=none]{d}[name=T, anchor=center]{}
    & \bZ_p \ar[rounded corners,
            to path={ -- ([xshift=2ex]\tikztostart.east)
                      |- (T.center) \tikztonodes
                      -| ([xshift=-2ex]\tikztotarget.west)
                      -- (\tikztotarget)}]{dll}[at end]{} \\   
&  \pi_{4k}(\mathrm{hofib}(\kappa)) \rar & K_{4k}(\bZ;\bZ_p) \rar
             \ar[draw=none]{d}[name=Z, anchor=center]{}
    & 0 \ar[rounded corners,
            to path={ -- ([xshift=2ex]\tikztostart.east)
                      |- (Z.center) \tikztonodes
                      -| ([xshift=-2ex]\tikztotarget.west)
                      -- (\tikztotarget)}]{dll}[at end]{} \\   
&  \pi_{4k-1}(\mathrm{hofib}(\kappa)) \rar & 0 \rar
             \ar[draw=none]{d}[name=X, anchor=center]{}
    & \bZ_p \ar[rounded corners,
            to path={ -- ([xshift=2ex]\tikztostart.east)
                      |- (X.center) \tikztonodes
                      -| ([xshift=-2ex]\tikztotarget.west)
                      -- (\tikztotarget)}]{dll}[at end]{} \\      
&  \pi_{4k-2}(\mathrm{hofib}(\kappa)) \rar & K_{4k-2}(\bZ;\bZ_p) \rar  & 0,
\end{tikzcd}
$$
so if $p$ satisfies the Vandiver conjecture (e.g.\ if $p \leq 2^{31}$) so that $K_{4k}(\bZ;\bZ_p)=0$ and $K_{4k-2}(\bZ;\bZ_p) \cong \bZ_p / \mathrm{Num}(B_{2k}/4k)$,  then $\pi_*(\mathrm{hofib}(\kappa))$ is determined by the maps $\kappa_{4k+1} : \bZ_p \cong K_{4k+1}(\bZ;\bZ_p) \to K_{4k+1}(\bZ_p;\bZ_p) \cong \bZ_p$ and extensions $0 \to \bZ_p \to \pi_{4k-2}(\mathrm{hofib}(\kappa)) \to \bZ_p / \mathrm{Num}(B_{2k}/4k) \to 0$. We do not know how to control these extensions, though \cite{BlumbergMandell} gives the expression $\pi_{4k-2}(\mathrm{hofib}(\kappa)) \cong H^1_{\et}(\bZ[\tfrac{1}{p}] ; \bZ/p^\infty(1-2k))^*$ for it which may be useful to some readers. Instead we shall restrict ourselves to work in the range of degrees where $\mathrm{Num}(B_{2k}/4k)$ are $p$-adic units, i.e.\ where $K_*(\bZ;\bZ_p)$ is torsion-free, in which case we just need to understand the maps $\kappa_{4k+1}$. The map $\kappa_{4k+1}$ is given by multiplication by the value $L_p(1+2k, \omega^{-2k})$ of the $p$-adic $L$-function, up to a $p$-adic unit; see \cite{LarsL} for a concise discussion. 

We may use this as follows. The $\bF_p$-cohomology of the fibre of 
$$\Omega^\infty\kappa_{4k+1} : K(K_{4k+1}(\bZ;\bZ_p), 4k+1) \lra K(K_{4k+1}(\bZ_p;\bZ_p), 4k+1)$$
in degrees $* < 4k+1 + 2p-2$ depends only on whether $L_p(1+2k, \omega^{-2k})$ is a $p$-adic unit: if so then the fibre has trivial $\bF_p$-cohomology, and if not then the fibre has $\bF_p$-cohomology $\bF_p[y_{4k}] \otimes \Lambda^*_{\bF_p}[y_{4k+1}]$ in this range.

By \cite[Proposition 11.3.12 (1)]{Cohen} we have
$$2k \cdot L_p(1+2k, \omega^{-2k}) = \lim_{r \to \infty} B_{\phi(p^r)-2k},$$
where the latter denote Bernoulli numbers. Assuming that $4k+1 < p$, so certainly $2k < p-1$, then as $\phi(p^r)-2k \equiv p-1-2k \mod p-1$ we have the Kummer congruences $\frac{B_{\phi(p^r)-2k}}{\phi(p^r)-2k} \equiv \frac{B_{p-1-2k}}{p-1-2k} \mod p$ and so
\begin{equation*}
L_p(1+2k, \omega^{-2k}) \dot{\equiv} B_{p-1-2k} \mod p,
\end{equation*}
where $\dot{\equiv}$ denotes congruence up to a $p$-adic unit. Thus we may determine whether or not $L_p(1+2k, \omega^{-2k})$ is a $p$-adic unit by calculating this residue class. 

Let us now treat the three primes 37, 16843, and 2124679 from the introduction. Certainly they all satisfy the Vandiver conjecture. For $p=37$ we find that $B_{p-1-2k} \not \equiv 0 \mod p$ for all $1+2k < p$ except for $1+2k=5$. This Bernoulli number contributes to torsion in $K_{62}(\bZ ; \bZ_p)$, so in degrees $* < 36$ there is no torsion in $K_*(\bZ; \bZ_p)$. It follows that
$$H^*(\mathrm{hofib}(\Omega^\infty \kappa);\bF_p) \cong \bF_p[x_2, x_6, x_{10}, \ldots] \otimes \bF_p[y_8] \otimes \Lambda^*_{\bF_p}[y_9]$$
in degrees $* < 36$. With Theorem \ref{thm:mainCong} this justifies the first example in the introduction.

For $p=16843$ we find that $B_{p-1-2k} \not \equiv 0 \mod p$ for all $1+2k < p$ except for $1+2k=3$. This Bernoulli number contributes to torsion in $K_{33678}(\bZ ; \bZ_p)$, so in degrees $* < 16842$ there is no torsion in $K_*(\bZ; \bZ_p)$. This gives the second example in the introduction.

For $p=2124679$ we find that $B_{p-1-2k} \not \equiv 0 \mod p$ for all $1+2k < p$ except for $1+2k=3$ and $1+2k = 1422781$. The first Bernoulli number contributes to torsion in $K_{4249350}(\bZ;\bZ_p)$, and the second contributes to torsion in $K_{1403794}(\bZ;\bZ_p)$, so in degrees $* < 1403794$ there is no torsion in $K_*(\bZ; \bZ_p)$. This gives the third example in the introduction.

\subsection{Cohomology of $p$-adic analytic groups}\label{sec:AnalyticGps}

For $m \geq 1$ and $p$ odd the groups $\SL_n(\bZ_p, p^m)$ are uniform ($=$ uniformly powerful) pro-$p$-groups \cite[Theorem 5.2]{DDMS} and hence their continuous $\bF_p$-cohomology may be described as the exterior algebra on their first $\bF_p$-cohomology (cf. \cite[V.2.2.7.2]{Lazard}\footnote{This formulation needs to be unwrapped quite a bit. The first step is to realise that ``\emph{équi-p-valués}'' is ``uniformly powerful''.}, \cite[Theorem 5.1.5]{SymondsWeigel}). On the other hand the map
$$I + p^m A \mapsto A \mod p : \SL_n(\bZ_p, p^m) \lra sl_n(\bF_p)$$
induces an isomorphism $\SL_n(\bZ_p, p^m)/\SL_n(\bZ_p, p^{m+1}) \overset{\sim}\to sl_n(\bF_p)$ and this is the maximal $p$-elementary abelian quotient for $n \geq 2$ by \cite[Lemma 5.1]{DDMS}. Thus there is an identification $H^1(\SL_n(\bZ_p, p^m);\bF_p) \cong sl_n(\bF_p)^\vee$, from which we obtain:

\begin{proposition}\label{prop:CtsCoh}
The induced map
$$\Lambda_{\bF_p}^*[sl_n(\bF_p)^\vee] \lra H^*_\mathrm{cts}(\SL_n(\bZ_p, p^m);\bF_p)$$
is an isomorphism.\qed
\end{proposition}

\section{Proof of Theorem \ref{thm:mainCong}}

The map of fibrations of pro-spaces
\begin{equation*}
\begin{tikzcd}
\{B\SL_n(\bZ, p^{r+m})\}_r \dar \rar & \{B\SL_n(\bZ, p^r)\}_r \dar \\
B\SL_n(\bZ, p^m) \dar \rar & B\SL_n(\bZ) \dar\\
\{B\SL_n(\bZ/ p^{r+m}, p^m)\}_r \rar & \{B\SL_n(\bZ/ p^r)\}_r 
\end{tikzcd}
\end{equation*}
gives a map of spectral sequences
\begin{equation*}
\begin{tikzcd}
{^I}E_2^{s,t} \arrow[r, equals] &  H^s_\mathrm{cts}(\SL_n(\bZ_p) ; \widetilde{H}^t(\SL_n(\bZ))) \dar \arrow[Rightarrow, r] & H^{s+t}(\SL_n(\bZ)) \dar \\
{^{I\!I}}E_2^{s,t} \arrow[r, equals] &  H^s_\mathrm{cts}(\SL_n(\bZ_p, p^m) ; \widetilde{H}^t(\SL_n(\bZ) )) \arrow[Rightarrow, r] & H^{s+t}(\SL_n(\bZ, p^m) ).
\end{tikzcd}
\end{equation*}
Corollary \ref{cor:TransSS} describes the structure of the first spectral sequence. Namely, if $n$ is large enough then in degrees $* < 2p-2$ the groups $\widetilde{H}^*(\SL_n(\bZ))$ are stable and operated upon trivially by $\SL_n(\bZ_p)$ by the theorem of Calegari and Emerton \cite{CEStability}, and furthermore are freely generated as a graded-commutative $\bF_p$-algebra by classes $\{x_\alpha\}_{\alpha \in I}$ which are transgressive, i.e.\ which survive until ${^I}E_{|x_\alpha|+1}^{0, |x_\alpha|}$ and then satisfy ${^I}d_{|x_\alpha|+1}(x_\alpha) = y_\alpha$ for certain $y_\alpha \in H^{|x_\alpha|+1}_\mathrm{cts}(\SL_n(\bZ_p))$. The map of spectral sequences means that each $x_\alpha$ also transgresses in the second spectral sequence, and furthermore transgress to the image of the corresponding $y_{\alpha}$ under the restriction map
$H^*_\mathrm{cts}(\SL_n(\bZ_p) ) \to H^*_\mathrm{cts}(\SL_n(\bZ_p, p^m) )$. We will prove that this map is zero in degrees $0<*< p$, so that the second spectral sequence collapses in a range:

\begin{theorem}\label{thm:ResIsZero}
The restriction maps
$$H^*_\mathrm{cts}(\SL_n(\bZ_p) ) \lra H^*_\mathrm{cts}(\SL_n(\bZ_p, p^m) )$$
are zero in degrees $0 < * < p$, for all $m>0$, and all large enough $n$.
\end{theorem}

\begin{corollary}\label{cor:SScollapse}
The spectral sequence
$${^{I\!I}}E_2^{s,t} =  H^s_\mathrm{cts}(\SL_n(\bZ_p, p^m) ; \widetilde{H}^t(\SL_n(\bZ))) \Longrightarrow H^{s+t}(\SL_n(\bZ, p^m) )$$
collapses in total degrees $* < p-1$.
\end{corollary}
\begin{proof}
By the discussion at the beginning of this Section, in total degree $* < 2p-2$ we have 
$${^{I\!I}}E_2^{*,t} = H^t_\mathrm{cts}(\SL_n(\bZ_p, p^m)) \otimes S_{\bF_p}[x_\alpha \, | \, \alpha \in I]$$
with $x_\alpha \in {^{I\!I}}E_2^{0, |x_\alpha|}$, and $S_{\bF_p}[-]$ denoting the free graded-commutative $\bF_p$-algebra. The map of spectral sequences from ${^{I}}E_2^{s,t}$ shows that $x_\alpha$ is transgressive, and transgresses to a class in the image of $H^{|x_\alpha|+1}_\mathrm{cts}(\SL_n(\bZ_p) ) \to H^{|x_\alpha|+1}_\mathrm{cts}(\SL_n(\bZ_p, p^m) )$. By Theorem \ref{thm:ResIsZero} this map is zero if $|x_\alpha|+1 < p$, so $x_\alpha$ is a permanent cycle if $|x_\alpha| < p-1$. By the Leibniz rule, the spectral sequence then collapses in this range.
\end{proof}

\subsection{Proof of Theorem \ref{thm:ResIsZero} for $m>1$}\label{sec:ResIsZeroMGeq1}

The proof in this case is essentially trivial. There is a factorisation
$$H^*_\mathrm{cts}(\SL_n(\bZ_p) ) \lra H^*_\mathrm{cts}(\SL_n(\bZ_p, p) ) \lra H^*_\mathrm{cts}(\SL_n(\bZ_p, p^m) )$$
of the restriction map and by the discussion in Section \ref{sec:AnalyticGps} we know the latter two cohomology rings: both are given by $\Lambda_{\bF_p}^*[sl_n(\bF_p)^\vee]$. However, the composition
$$\SL_n(\bZ_p, p^m) \lra \SL_n(\bZ_p, p) \lra \SL_n(\bZ_p, p)/\SL_n(\bZ_p, p^2) = sl_n(\bF_p)$$
is trivial for $m>1$, so the second map in the factorisation of the restriction map is trivial in degrees $*>0$, and so the restriction map is too.

\subsection{Proof of Theorem \ref{thm:ResIsZero} for $m=1$}\label{sec:ResIsZeroMEq1}

The fibration of pro-spaces
\begin{equation}\label{eq:EdgeSS}
\{B\SL_n(\bZ/ p^{r+1}, p)\}_r \lra  \{B\SL_n(\bZ/ p^r)\}_r \lra B\SL_n(\bZ/p)
\end{equation}
yields a spectral sequence
$$^{I\!I\!I}E^{s,t}_2 = H^s(\SL_n(\bZ/p) ; {H}^t_\mathrm{cts}(\SL_n(\bZ_p, p) )) \Longrightarrow H^{s+t}_\mathrm{cts}(\SL_n(\bZ_p) ),$$
and the map of Theorem \ref{thm:ResIsZero} factors as
$$H^*_\mathrm{cts}(\SL_n(\bZ_p) ) \overset{\text{edge hom.}}\lra H^0(\SL_n(\bZ/p) ; {H}^*_\mathrm{cts}(\SL_n(\bZ_p, p) )) \overset{\text{inc.}}\lra H^*_\mathrm{cts}(\SL_n(\bZ_p, p)),$$
so we must show that this edge homomorphism is trivial for $0 < * < p$. The following proposition describes the $d_2$-differential on this edge.

\begin{proposition}\label{prop:SSbehaviour}
In degrees $* < p$ and for all large enough $n$ the differential
$$d_2 : {^{I\!I\!I}E}^{0,*}_2 \lra {^{I\!I\!I}E}^{2,*-1}_2 $$
is given by
\begin{align}
\label{eq:derivation}
\begin{split}
\Lambda^*_{\bF_p}[c_3, c_5, c_7, \ldots] &\lra \bF_p\{e_1, e_2, e_3, \ldots\} \otimes \Lambda^*_{\bF_p}[c_3, c_5, c_7, \ldots]\\
c_t &\longmapsto t \cdot e_{t-1} \otimes 1
\end{split}
\end{align}
and the Leibniz rule, where $c_t$ and $e_t$ have degree $t$.
\end{proposition}

We defer the (quite involved) proof of this proposition to Section \ref{sec:PfPropSSbehaviour}. 

To finish the proof of Theorem \ref{thm:ResIsZero} in the case $m=1$, we claim that the map \eqref{eq:derivation} is injective in degrees $0 < * < p$. To see this, suppose that $x \in \Lambda^*_{\bF_p}[c_3, c_5, c_7, \ldots]$ has $d_2(x)=0$ and $0<|x|<p$. Then, for each $i \in \{3,5,7,\ldots\}$ with $i<p$ write $x = A_i + c_i \cdot B_i$ with $A_i$ and $B_i$ not containing $c_i$, so that
\begin{align*}
0 = d_2(x) &= d_2(A_i+c_i \cdot B_i)\\
&= d_2(A_i) + i \cdot e_{i-1} \otimes B_i - c_i \cdot d_2(B_i).
\end{align*}
As $A_i$ and $B_i$ do not contain $c_i$, $d_2(A_i)$ and $d_2(B_i)$ do not contain terms of the form $e_{i-1} \otimes ?$ and so cannot cancel with the middle term: thus $i \cdot B_i = 0$, and as $i < p$ it follows that $B_i=0$. Thus $c_i$ does not occur in $x$, but this goes for all $i$, so $x=0$. Now, by Proposition \ref{prop:SSbehaviour} this means that the differential $d_2 : {^{I\!I\!I}E}^{0,*}_2 \to {^{I\!I\!I}E}^{2,*-1}_2$
is injective for $0 < * < p$, and hence that the edge homomorphism
$$H^{*}_\mathrm{cts}(\SL_n(\bZ_p) ) \lra H^0(\SL(\bZ/p) ; {H}^*_\mathrm{cts}(\SL_n(\bZ_p, p) ))$$
is trivial in degrees $0 < * < p$, proving Theorem \ref{thm:ResIsZero}.

\subsection{Resolving extensions}\label{sec:Ext}

By Corollary \ref{cor:SScollapse} the lowest bidegree in which there could be a differential in the second spectral sequence is $d_p : {^{I\!I}}E_p^{0,p-1} \to {^{I\!I}}E_p^{p,0}$. It follows that the associated graded of the Hoschchild--Leray--Serre filtration on $H^{*}(\SL_n(\bZ, p^m))$ satisfies
\begin{equation}\label{eq:AssocGr}
\mathrm{Gr}^\bullet H^{*}(\SL_n(\bZ, p^m) ) \cong \Lambda^\bullet_{\bF_p}[sl_n(\bF_p)^\vee] \otimes \widetilde{H}^{*-\bullet}(\SL_n(\bZ) )
\end{equation}
in degrees $* < p-1$, as bigraded $\bF_p$-algebras and as $\SL(\bZ/p^m)$-representations. To prove Theorem \ref{thm:mainCong} we must show that $H^{*}(\SL_n(\bZ, p^m)) \cong \bigoplus_s \mathrm{Gr}^s H^{*}(\SL_n(\bZ, p^m))$, as graded $\bF_p$-algebras and $\SL_n(\bZ/p^m)$-representations, in degrees $* < p-1$. Using the multiplicative structure, to do so it suffices to show that the quotient map
$$H^{*}(\SL_n(\bZ, p^m)) \lra \mathrm{Gr}^0 H^{*}(\SL_n(\bZ, p^m)) = \widetilde{H}^{*}(\SL_n(\bZ))$$
is split as a map of $\bF_p$-algebras and $\SL_n(\bZ/p^m)$-representations, in degrees $* < p-1$.

As $\widetilde{H}^{*}(\SL_n(\bZ))$ is a free graded-commutative $\bF_p$-algebra on certain classes $\{x_\alpha\}_{\alpha \in I}$, by part of Corollary \ref{cor:TransSS}, it suffices to show that for each of the free generators $x_\alpha \in \widetilde{H}^*(\SL_n(\bZ) )$ of degree $|x_\alpha| < p-1$ there exists an $\SL_n(\bZ/p^m)$-invariant element $\bar{x}_{\alpha} \in H^{*}(\SL_n(\bZ, p^m) )$ which restricts to 
$$x_{\alpha} \in \frac{H^{|x_\alpha|}(\SL_n(\bZ, p^m))}{F^{|x_\alpha|-1} H^{|x_\alpha|}(\SL_n(\bZ, p^m))} =  H^0_\mathrm{cts}(\SL_n(\bZ_p, p^m) ; \widetilde{H}^{|x_\alpha|}(\SL_n(\bZ) )).$$
By pulling back along the natural map $\SL_n(\bZ_p, p^m) \to \SL_n(\bZ_p, p)$, it suffices to produce such $\bar{x}_{\alpha}$'s in the case $m=1$

To do so, consider the filtration of the $\bF_p[\SL_n(\bZ/p)]$-module $H^{|x_\alpha|}(\SL_n(\bZ, p) )$, whose associated graded is
$$\mathrm{Gr}^t H^{|x_\alpha|}(\SL_n(\bZ, p) ) \cong \Lambda^t_{\bF_p}[sl_n(\bF_p)^\vee] \otimes \widetilde{H}^{|x_\alpha|-t}(\SL_n(\bZ) ),$$
where the second factor has trivial $\SL_n(\bZ/p)$-action. The spectral sequence of this filtered module takes the form
$$^{I\!V}E^{s,t}_1 = H^s(\SL_n(\bZ/p) ; \mathrm{Gr}^t H^{|x_\alpha|}(\SL_n(\bZ, p) )) \Longrightarrow H^{s}(\SL_n(\bZ/p);H^{|x_\alpha|}(\SL_n(\bZ, p) ))$$
$$d_r : {^{I\!V}E^{s,t}_r} \lra {^{I\!V}E^{s+1,t+r}_r}.$$

\begin{lemma}
As long as $|x_\alpha| < p-1$ we have $^{I\!V}E^{1,*}_1=0$, for all large enough $n$.
\end{lemma}
\begin{proof}
As an $\SL_n(\bZ/p)$-representation $\mathrm{Gr}^t H^{|x_\alpha|}(\SL(\bZ, p) )$ is a direct sum of copies of $\Lambda^t_{\bF_p}[sl_n(\bF_p)^\vee]$, so we must show that these representations have trivial first cohomology for $t < p-1$. As in Section \ref{sec:Semisimplicity}, by homological stability we may suppose that $n \not\equiv 0 \mod p$, so that with $V = \bF_p^n$ the standard representation the sequence $0 \to \bF_p \overset{coev}\to V \otimes V^\vee \to sl_n(\bF_p)^\vee \to 0$ has a splitting given by $\tfrac{1}{n} ev : V \otimes V^\vee \to \bF_p$, and hence $sl_n(\bF_p)^\vee \oplus \bF_p \cong V \otimes V^\vee$. Then $\Lambda^t_{\bF_p}[sl_n(\bF_p)^\vee]$ is a summand of $\Lambda^t_{\bF_p}[V \otimes V^\vee]$ so it suffices to show that the latter representation has trivial first cohomology. On the other hand, as $t < p$ we have that $\Lambda_{\bF_p}^t[V \otimes V^\vee]$ is a summand of $(V \otimes V^\vee)^{\otimes t}$, so it suffices to show that  the latter representation has trivial first cohomology. This follows from Theorem \ref{thm:StabCohTensorPowers}, bearing in mind that the class $x$ in the domain of the map \eqref{eq:FF} has degree 2.
\end{proof}

It follows that for $|x_\alpha| < p-1$ the class $x_{\alpha} \in H^0(\SL_n(\bZ/p) ; \mathrm{Gr}^{0} H^{|x_\alpha|}(\SL_n(\bZ, p) )) = {^{I\!V}E^{0,0}_1}$ is a permanent cycle in this spectral sequence, so that we may find a $\bar{x}_{\alpha} \in H^{0}(\SL_n(\bZ/p);H^{|x_\alpha|}(\SL_n(\bZ, p) ))$ restricting to $x_\alpha$, as required.

\begin{remark}[Addenda]
The above argument proves something a little stronger than just Theorem \ref{thm:mainCong}. 

Firstly the argument of Section \ref{sec:ResIsZeroMGeq1} only used that $0 < *$, giving \eqref{eq:AssocGr} in degrees $* < 2p-2$ for $m>1$ and all large enough $n$. The right-hand side of \eqref{eq:AssocGr} is a free graded-commutative algebra in degrees $* < 2p-2$, so the multiplicative extensions are trivial. However, triviality of the extensions as $\SL_n(\bZ/p^m)$-representations were obtained by comparison with the case $m=1$, so we only know them for $* < p-1$.

Secondly, in the case $m=1$ Corollary \ref{cor:SScollapse} gives the identity \eqref{eq:AssocGr} not only in degrees $* < p-1$ but also for $*=p-1$ and $\bullet > 0$. Furthermore it gives an exact sequence
$$0 \to \mathrm{Gr}^0 H^{p-1}(\SL_n(\bZ, p)) \to {^{I\!I}}E_p^{0,p-1} = \widetilde{H}^{p-1}(\SL_n(\bZ) ) \overset{d_p}\to {^{I\!I}}E_p^{p,0} = \Lambda^p_{\bF_p}[sl_n(\bF_p)^\vee].$$
Combined with the same line of reasoning as above, this can be used to also obtain an injection
$$H^{p-1}(\SL_n(\bZ, p)) \lra \bigoplus_{a+b=p-1}\Lambda^{a}_{\bF_p}[sl_n(\bF_p)^\vee] \otimes \widetilde{H}^{b}(\SL_n(\bZ) )$$
of $\SL_n(\bZ/p)$-representations, compatible with the evident multiplicative structures in degrees $* \leq  p-1$.
\end{remark}

\section{Proof of Proposition \ref{prop:SSbehaviour}}\label{sec:PfPropSSbehaviour}

\subsection{Tensor powers of the dual adjoint representation}
Recall that we write $sl(V)$ for the kernel of $ev : V \otimes V^\vee \to \bF_p$, so there is a short exact sequence
\begin{equation}\label{eq:Res}
0 \lra \bF_p \overset{coev}\lra V \otimes V^\vee \lra sl(V)^\vee \lra 0.
\end{equation}

In Sections \ref{sec:FunctorHomology} and \ref{sec:wBr} we have constructed maps 
$$\psi_{S,T} : \Gamma_{\bF_p}[x]^{\otimes S} \otimes k\{\mathrm{Bij}(T,S)\} \overset{\Psi_{S,T}}\lra \Ext_{\GL}^*(I^{\otimes T}, I^{\otimes S}) \lra \Ext_{\GL(V)}^*(V^{\otimes T}, V^{\otimes S})$$
and shown that $\Psi_{S,T}$ is an isomorphism (in Theorem \ref{thm:StabCohTensorPowers}) and that the second map is an isomorphism in a stable range of degrees (by the stability discussion in Section \ref{sec:stab}). Furthermore, the latter may be written as $H^*(\GL(V) ; V^{\otimes S} \otimes (V^\vee)^{\otimes T})$, and as discussed in Section \ref{sec:SLvsGL} we may replace $\GL(V)$ by $\SL(V)$ in a stable range.

Using this we may form the composition
\begin{equation}\label{eq:Contract}
\begin{aligned}
\Gamma_{\bF_p}[x]^{\otimes T} \otimes \bF_p\{\Sigma^{ad}_T\} &\overset{\psi_{T,T}}\lra H^*(\SL(V) ; (V \otimes V^\vee)^{\otimes T})\\
&\lra H^*(\SL(V) ; (sl(V)^\vee)^{\otimes T})
\end{aligned}
\end{equation}
for which we have the following analogue of Theorem \ref{thm:StabCohTensorPowers}.

\begin{lemma}\label{lem:CohAdRep}
In a stable range of degrees, the kernel of \eqref{eq:Contract} is spanned by those $\big( \bigotimes_{t \in T} x_t^{[\ell(t)]} \big) \otimes \sigma$ such that there is a $j \in T$ with $\ell(j)=0$ and $\sigma(j)=j$.
\end{lemma}
\begin{proof}
By the results of Section \ref{sec:Stability}, we may increase the dimension of $V$ if we wish, and in particular we may suppose that $\dim(V) \not\equiv 0 \mod p$, so that \eqref{eq:Res} is $\GL(V)$-equivariantly split by a unit times the map $ev : V \otimes V^\vee \to \bF_p$. Consider $\bF_p \overset{coev}\to V \otimes V^\vee$ as a chain complex whose homology is $sl(V)^\vee$ supported in degree 0. The $T$-th tensor power of this chain complex has the form
$$\cdots \lra \bigoplus_{j \in T} (V \otimes V^\vee)^{\otimes T \setminus j} \lra (V \otimes V^\vee)^{\otimes T},$$
where the rightmost map is given by inserting $coev$ on each of the $T$ summands. By the K{\"u}nneth theorem the homology of this complex is $(sl(V)^\vee)^{\otimes T}$ supported in degree zero, giving an exact sequence
$$\bigoplus_{j \in T} (V \otimes V^\vee)^{\otimes T \setminus j} \lra (V \otimes V^\vee)^{\otimes T} \lra (sl(V)^\vee)^{\otimes T} \lra 0.$$
We have arranged that the middle map is $\GL(V)$-equivariantly split, so this sequence remains exact after applying $H^*(\SL(V) ; -)$. In a stable range of degrees the map $\psi_{T,T}$ is an isomorphism by Theorem \ref{thm:StabCohTensorPowers} and the discussion after it, as are the maps $\psi_{T \setminus j, T \setminus j}$. Thus after applying $H^*(\SL(V) ; -)$ we obtain an exact sequence
$$\bigoplus_{j \in T} \Gamma_{\bF_p}[x]^{\otimes T\setminus j} \otimes \bF_p\{\Sigma^{ad}_{T \setminus j}\} \to \Gamma_{\bF_p}[x]^{\otimes T} \otimes \bF_p\{\Sigma^{ad}_T\} \overset{\text{\eqref{eq:Contract}}}\to H^*(\SL(V) ; (sl(V)^\vee)^{\otimes T}) \to 0$$
where the left-hand map is described by the functoriality on the upward walled Brauer category, as in Section \ref{sec:uwBr}. Namely, on the $j$th summand it is induced by the map $(inc, inc, \mathrm{Id}_{\{j\}}) : (T \setminus j, T \setminus j) \to (T,T)$ in the upward walled Brauer category, which we defined to send $(\bigotimes_{s \in T \setminus j} x_s^{[\ell'(s)]}) \otimes \sigma'$ to $(\bigotimes_{t \in T} x_t^{[\ell(t)]}) \otimes \sigma$, where $\sigma \in \Sigma^{ad}_T$ fixes $j \in T$  and is given by $\sigma'$ on $T \setminus j$, and $\ell$ sends $j$ to 0 and agrees with $\ell'$ on $T \setminus j$. The image of this map is spanned by the claimed elements. 
\end{proof}

\begin{corollary}\label{cor:CalcGps}
There are classes
\begin{align*}
c_t \in H^0(\SL(V) ; \Lambda^t[sl(V)^\vee]) & \text{ for $t \geq 3$ odd}\\
e_t \in H^2(\SL(V) ; \Lambda^t[sl(V)^\vee]) & \text{ for $t \geq 1$}
\end{align*}
and isomorphisms
\begin{align*}
H^0(\SL(V) ; \Lambda^*[sl(V)^\vee]) &\cong \Lambda^*_{\bF_p}[c_3, c_5, c_7, \ldots]\\
H^1(\SL(V) ; \Lambda^*[sl(V)^\vee]) &\cong 0\\
H^2(\SL(V) ; \Lambda^*[sl(V)^\vee]) &\cong \bF_p\{e_1, e_2, e_3, \ldots\} \otimes \Lambda^*_{\bF_p}[c_3, c_5, c_7, \ldots]
\end{align*}
in degrees $* < p$ for $\dim(V)$ sufficiently large.
\end{corollary}
\begin{proof}
Recall that we write $\ul{t} := \{1,2,\ldots, t\}$. Let $\tilde{c}_t \in H^0(\SL(V) ; \Lambda^t_{\bF_p}[V \otimes V^\vee])$ denote the image of the class 
$$1 \otimes (1,2,3,\ldots, t) \in \Gamma_{\bF_p}[x]^{\otimes \ul{t}} \otimes \bF_p\{\Sigma^{ad}_{\ul{t}}\}$$
under the maps
\begin{equation}\label{eq:push}
\Gamma_{\bF_p}[x]^{\otimes \ul{t}} \otimes \bF_p\{\Sigma^{ad}_{\ul{t}}\} \to H^*(\SL(V) ; (V \otimes V^\vee)^{\otimes \ul{t}}) \to H^*(\SL(V) ; \Lambda^t_{\bF_p}[V \otimes V^\vee])
\end{equation}
and $c_t$ its further image in $H^*(\SL(V) ; \Lambda^t_{\bF_p}[sl(V)^\vee])$. The discussion above shows that $c_1=0$. Writing $\bF_p^-$ for the sign representation, for $t<p$ we have
$$H^0(\SL(V) ; \Lambda^t_{\bF_p}[V \otimes V^\vee]) \cong \bF_p\{\Sigma^{ad}_{\ul{t}}\} \otimes_{\Sigma_{\ul{t}}} \bF_p^-.$$
Recalling that we have assumed that $p$ is odd, this is given by the way conjugacy classes split in the alternating group: if a conjugacy class contains an even cycle or two odd cycles of the same length, then it becomes trivial on applying $- \otimes_{\Sigma_{\ul{t}}} \bF_p^-$; otherwise it contributes a 1-dimensional space. This identifies this space with the degree $t$ part of $\Lambda^*_{\bF_p}[\tilde{c}_1, \tilde{c}_3, \tilde{c}_5, \ldots]$, and hence identifies $H^0(\SL(V) ; \Lambda^*_{\bF_p}[V \otimes V^\vee])$ with this graded-commutative algebra.

As $\Lambda^t_{\bF_p}[sl(V)^\vee]$ is a summand of $(sl(V)^\vee)^{\otimes \ul{t}}$ for $t < p$, using Lemma \ref{lem:CohAdRep} we obtain the claimed formula for $H^0(\SL(V) ; \Lambda^*_{\bF_p}[sl(V)^\vee])$, and also for $H^1(\SL(V) ; \Lambda^*_{\bF_p}[sl(V)^\vee])$ because $H^*(\SL(V) ; (V \otimes V^\vee)^{\otimes \ul{t}})$ is supported in even degrees in the stable range by Theorem \ref{thm:StabCohTensorPowers}.

Let $\tilde{e}_t \in H^2(\SL(V) ; \Lambda^t_{\bF_p}[V \otimes V^\vee])$ denote the image of the class 
$$x_1^{[1]} \otimes (1,2,3,\ldots, t) \in \Gamma_{\bF_p}[x]^{\otimes \ul{t}} \otimes \bF_p\{\Sigma^{ad}_{\ul{t}}\}$$
under the maps \eqref{eq:push} and $e_t$ its further image in $H^*(\SL(V) ; \Lambda^t_{\bF_p}[sl(V)^\vee])$.  We have
$$H^2(\SL(V) ; \Lambda^t_{\bF_p}[V \otimes V^\vee]) \cong \left(\bF_p\{x_i^{[1]} \, | \, i \in \ul{t}\} \otimes \bF_p\{\Sigma^{ad}_{\ul{t}}\}\right) \otimes_{\Sigma_{\ul{t}}} \bF_p^-.$$
We think of the first factor as being the space of permutations of $\ul{t}$ written as disjoint cycles, with one entry marked, on which $\Sigma_{\ul{t}}$ acts by conjugation. As above we find that to contribute the unmarked cycles must all be of different odd lengths, and the marked cycle may be of any length. As a module over $H^0(\SL(V) ; \Lambda^*_{\bF_p}[V \otimes V^\vee])$ we therefore have that $H^2(\SL(V) ; \Lambda^*_{\bF_p}[V \otimes V^\vee])$ is free on the basis of marked cycles, i.e.\ $\tilde{e}_1, \tilde{e}_2, \tilde{e}_3, \ldots$. Using as above that $\Lambda^t_{\bF_p}[sl(V)^\vee]$ is a summand of $(sl(V)^\vee)^{\otimes \ul{t}}$ for $t < p$, we find the same description for $H^2(\SL(V) ; \Lambda^*_{\bF_p}[sl(V)^\vee])$ as a module over $H^0(\SL(V) ; \Lambda^*_{\bF_p}[sl(V)^\vee])$.
\end{proof}

\subsection{Multiplicative structure on the $E_2$-page of the Serre spectral sequence}\label{sec:SerreSS}

We will need to use some details of the $d_2$-differentials in the spectral sequence for a group extension
\begin{equation}\label{eq:GpExt}
1 \lra K \lra G \lra Q \lra 1,
\end{equation}
with coefficients in a field $\bk$: it has the form
$$E_2^{s,t} = H^s(Q ; H^t(K ; \bk)) = \Ext^s_{\bk[Q]}(\bk, H^t(K;\bk)) \Longrightarrow H^{s+t}(G;\bk).$$

\begin{lemma}\label{lem:d2}
There are canonical elements $d_2^t \in \Ext_{\bk[Q]}^2(H^t(K;\bk), H^{t-1}(K;\bk))$  for $t \geq 1$ such that:
\begin{enumerate}[(i)]
\item The differential $d_2 : E_2^{s, t} \to E_2^{s+2, t-1}$ is given by the Yoneda product with $d_2^t$.

\item The element $-d_2^1 \in \Ext_{\bk[Q]}^2(H^1(K;\bk), \bk) \cong H^2(Q ; H_1(K;\bk))$ classifies the extension obtained from \eqref{eq:GpExt} by pushout along $K \overset{\text{ab}}\to H_1(K;\bZ) \to H_1(K;\bk)$.

\item The square
\begin{equation*}
\begin{tikzcd}[column sep = 5em]
H^{t'}(K;\bk) \otimes H^{t''}(K;\bk) \rar{- \smile-} \arrow[d, "d_2^{t'} \otimes \mathrm{Id} + (-1)^{t'} \mathrm{Id} \otimes d_2^{t''}"] & H^{t'+ t''}(K;\bk) \arrow[d, "d_2^{t'+t''}"] \\
\parbox{4.5cm}{$H^{t'-1}(K;\bk) \otimes H^{t''}(K;\bk)[2] \oplus H^{t'}(K;\bk) \otimes H^{t''-1}(K;\bk)[2]$} \rar{(- \smile -)[2]}& H^{t'+ t''-1}(K;\bk)[2]
\end{tikzcd}
\end{equation*}
commutes in the derived category of $\bk[Q]$-modules.
\end{enumerate}
\end{lemma}

We were surprised to not be able to find this result in the literature, although the analogue of (i) and (ii) in homology has been developed in various works of Legrand, see e.g.\ \cite{Legrand}, and is discussed for the change-of-rings spectral sequence by Suárez-Alvarez \cite[Theorem 2.2.3]{MSA}.

\begin{proof}[Proof of Lemma \ref{lem:d2}]
Let $E_*G \to \bk$ be the standard free resolution of the trivial $\bk[G]$-module, which in homological degree $p$ is $\bk[G^p]$, and consider the chain complex
$$C := \mathrm{Hom}_{\bk[K]}(E_*G, \bk) = [\mathrm{Hom}_\bk(E_*G, \bk)]^K,$$
which has a residual action of $G/K=Q$. The homology of this complex in degree $-i$ is $H^i(K;\bk)$, with the $Q$-action induced by the outer $Q$-action on $K$ coming from the extension \eqref{eq:GpExt}.

Deconcatenation gives a map of $\bk[G]$-modules 
\begin{align*}
\Delta : E_*G &\lra E_*G \otimes_\bk E_*G\\
(g_1, \ldots, g_p) &\longmapsto \sum_{i=0}^{p} (g_1, \ldots, g_i) \otimes (g_{i+1}, \ldots, g_p)
\end{align*}
which yields a morphism $\phi \otimes \psi \mapsto (\phi \otimes \psi) \circ \Delta : C \otimes_\bk C \to C$ of chain complexes over $\bk[Q]$. Precomposition with $E_*G \to \bk$ gives a morphism $\bk \to C$ of chain complexes over $\bk[Q]$. On homology this gives the cup-product on $H^*(K;\bk)$, and its unit.

Using this structure, we consider $C$ to be a unital associative ring object in the derived category $\mathsf{D}(\bk[Q])$ of $\bk[Q]$, which we equip with the natural $t$-structure and corresponding truncation functors denoted $\tau$, as well as the symmetric monoidal structure given by $- \otimes_\bk -$ (with the diagonal $Q$-action). The monoidal product of connective objects is connective, from which it follows that the monoidal product of an $a$-connective and a $b$-connective object is $(a+b)$-connective: thus the multiplication and unit on $C$ gives morphisms $\tau_{\geq -i} C \otimes_\bk \tau_{\geq -j} C \to \tau_{\geq -i-j}C$ and $\bk \to \tau_{\geq 0} C$.

The distinguished triangles for adjacent connective covers of $C$ yield its Whitehead tower, having the form
\begin{equation*}
\begin{tikzcd}[row sep=small]
 & 0 \dar\\
H^1(K;\bk)[-2]  \rar & \tau_{\geq 0} C \dar \rar{\sim} & H^0(K;\bk)[0]\simeq \bk\\
H^2(K;\bk)[-3] \rar & \tau_{\geq -1} C \dar \rar & H^1(K;\bk)[-1] \\
H^3(K;\bk)[-4] \rar & \tau_{\geq -2} C \rar\dar & H^2(K;\bk)[-2] \\
 & \vdots
\end{tikzcd}
\end{equation*}
The horizontal compositions represent classes
\begin{equation*}
\begin{aligned}
d_2^t \in & \,\, \Ext^2_{\bk[Q]}(H^t(K;\bk), H^{t-1}(K;\bk)) \\
&= [H^t(K;\bk), H^{t-1}(K;\bk)[2]]_{\mathsf{D}(\bk[Q])}\\
& \cong [H^t(K;\bk)[-t-1], H^{t-1}(K;\bk)[1-t]]_{\mathsf{D}(\bk[Q])}.
\end{aligned}
\end{equation*}
Applying $[\bk, -]_{\mathsf{D}(\bk[Q])}$ to this diagram yields an exact couple and so a spectral sequence: this is the (Lyndon--Hochschild--)Serre spectral sequence for the group extension \eqref{eq:GpExt}, with
$$E_2^{s,t} := [\bk[-s-t], H^t(K;\bk)[-t]]_{\mathsf{D}(\bk[Q])} \cong \Ext^s_{\bk[Q]}(\bk, H^t(K;\bk))$$
and converging to $[\bk[-s-t], C]_{\mathsf{D}(\bk[Q])} \cong H^{s+t}(G;\bk)$. The $d_2$-differentials are tautologically given by Yoneda product with the $d_2^t$'s, giving (i).

For (ii), note that if $M$ is a $\bk[Q]$-module then the filtered object $(\tau_{\geq \bullet} C) \otimes_\bk M$ yields the (Lyndon--Hochschild--)Serre spectral sequence with coefficients in $M$,
$$E_2^{s,t} = H^s(Q; H^t(K;\bk)\otimes_\bk M) \Longrightarrow H^{s+t}(G;M).$$
Thus consider the map between spectral sequences with $M=H_1(K;\bk)$ induced by the pushout extension
\begin{equation*}
\begin{tikzcd}[row sep=small]
1 \rar & K \rar \dar & G \rar \dar & Q \rar \arrow[d, equals]& 1\\
1 \rar & H_1(K;\bk) \rar & E \rar & Q \rar & 1.
\end{tikzcd}
\end{equation*}
Writing $coev \in H^1(H_1(K;\bk) ; \bk) \otimes_\bk H_1(K;\bk) = \mathrm{Hom}_\bZ(H_1(K;\bk) , \bk) \otimes_\bk H_1(K;\bk)$ for the class dual to evaluation, its image under the differential 
$$d_2 : H^0(Q ; H^1(H_1(K;\bk) ; \bk) \otimes H_1(K;\bk) ) \lra H^2(Q ; H_1(K;\bk) )$$
in the spectral sequence for the lower extension is \emph{minus} the class classifying the lower extension \cite[Theorem 4]{HS}. By naturality this class is the composition
$$\bk[-2] \overset{coev}\lra H^1(K;\bk) [-2] \otimes_\bk H_1(K;\bk) \overset{d_2^1 \otimes H_1(K;\bk) }\lra \bk \otimes_\bk H_1(K;\bk),$$
which corresponds to the class $d_2^1$ under the isomorphism $\Ext_{\bk[Q]}^2(H^1(K;\bk), \bk) \cong \Ext_{\bk[Q]}^2(\bk, H_1(K;\bk)) = H^2(Q ; H_1(K;\bk))$.

For (iii) we appeal to the fact that $\mathsf{D}(\bk[Q])$ has an enhancement to the unbounded derived $\infty$-category as constructed in \cite[\S 1.3.5]{HA} by localising the category $\mathsf{Ch}(\bk[Q])$ of chain complexes at the quasi-isomorphisms. This $\infty$-category is stable \cite[Proposition 1.3.5.9]{HA}, and the $t$-structure we have been using refines to an $\infty$-categorical $t$-structure \cite[\S 1.2.1]{HA}. Furthermore, $- \otimes_\bk -$ on $\mathsf{Ch}(\bk[Q])$ induces a symmetric monoidal structure, compatible with the $t$-structure in the sense that the monoidal product of connective objects is connective. Placing ourselves in this setting, we may rely on Hedenlund's thesis \cite[Part II]{HedenlundThesis}. Namely, $C$ defines a (commutative) algebra object in the derived $\infty$-category of $\bk[Q]$, and the Whitehead tower gives a multiplicative filtration of it, so by \cite[Theorem II.1.21]{HedenlundThesis} we have a Leibniz rule in the form of the commutative diagram
\begin{equation*}
\begin{tikzcd}[column sep = 4em]
H^{t'}(K;\bk)[-t'] \otimes H^{t''}(K;\bk)[-t''] \arrow[r, "- \smile -"] \arrow[d, "{d_2^{t'}[-t']} \otimes \mathrm{Id} + \mathrm{Id} \otimes {d_2^{t''}[-t'']}"] & H^{t'+t''}(K;\bk)[-t'-t''] \arrow[d, "{d_2^{t'+t''}[-t'-t'']}"]\\
\parbox{6.5cm}{$H^{t'-1}(K;\bk)[2-t'] \otimes  H^{t''}(K;\bk)[-t''] \oplus H^{t'}(K;\bk)[-t']  \otimes H^{t''-1}(K;\bk)[2-t'']$} \rar{- \smile -}
 & H^{t'+t''-1}(K;\bk)[2-t'-t'']
\end{tikzcd}
\end{equation*}
in $\mathsf{D}(\bk[Q])$. There is a subtlety about signs. Shifting is implemented by $\bk[1] \otimes_\bk -$, so the bottom cup-product map on the second summand actually requires commuting $\bk[1]$ past $H^{t'}(K;\bk)[-t']$ before cupping, incurring a sign $(-1)^{t'}$. When we shift up by $t'+t''$ in order to put this square in the form stated in the lemma, this can be expressed by writing the left-hand vertical map as $d_2^{t'} \otimes \mathrm{Id} + (-1)^{t'} \mathrm{Id} \otimes d_2^{t''}$, and the bottom map as just cupping.
\end{proof}

\begin{remark}
For readers wishing to completely avoid $\infty$-categories, the proof of \cite[Theorem II.1.21]{HedenlundThesis} and the results leading up to it can be implemented in the tensor triangulated category $\mathsf{D}(\bk[Q])$ by making use of the modest additional axioms (TC3)-(TC5) of \cite{MayTriangulated}, which hold in this setting as $\mathsf{D}(\bk[Q])$ is the homotopy category of a monoidal model category satisfying the pushout-product axiom.
\end{remark}

\begin{remark}
For the Serre spectral sequence of a homotopy fibre sequence $F \to E \overset{\pi}\to B$, the analogous conclusion may be obtained by replacing $\mathsf{D}(\bk[Q])$ with the symmetric monoidal stable $\infty$-category $H_B\bk\text{-}\mathsf{mod}$ of modules in parameterised spectra $\mathsf{Sp}_B$ over the constant Eilenberg--MacLane spectrum $H_B\bk$. Then $C := F_B(\Sigma^\infty_B E, H_B\bk)$ is a ring object via the fibrewise diagonal maps $E \to E \times_B E$. Its Whitehead tower gives rise to the Serre spectral sequence for $\pi : E \to B$ and can be analysed in parallel to the above by again invoking Hedenlund's thesis \cite{HedenlundThesis}.
\end{remark}

\subsection{Proof of Proposition \ref{prop:SSbehaviour}}

Recall that we wish to understand the differentials $d_2 : {^{I\!I\!I}E}^{0,*}_2 \to {^{I\!I\!I}E}^{2,*-1}_2$ in the spectral sequence
$$^{I\!I\!I}E^{s,t}_2 = H^s(\SL_n(\bZ/p) ; {H}^t_\mathrm{cts}(\SL_n(\bZ_p, p))) \Longrightarrow H^{s+t}_\mathrm{cts}(\SL_n(\bZ_p) ),$$
associated to the fibration \eqref{eq:EdgeSS}. In Section \ref{sec:AnalyticGps} we explained how to identify $H^t_\mathrm{cts}(\SL_n(\bZ_p, p) ) = \Lambda^t_{\bF_p}[sl_n(\bF_p)^\vee]$ as an $\SL(\bZ/p)$-module, so by Corollary \ref{cor:CalcGps} we have calculated that
\begin{align*}
{^{I\!I\!I}E}^{0,*}_2 &\cong \Lambda^*_{\bF_p}[c_3, c_5, c_7, \ldots]\\
{^{I\!I\!I}E}^{2,*}_2 &\cong \bF_p\{e_1, e_2, e_3, \ldots\} \otimes \Lambda^*_{\bF_p}[c_3, c_5, c_7, \ldots]
\end{align*}
for $* < p$ and all large enough $n$. To evaluate the differential we apply the discussion in Section \ref{sec:SerreSS}. That discussion provides elements
\begin{align*}
d_2^t &\in \mathrm{Ext}^2_{\bF_p[\SL_n(\bZ/p)]}(\Lambda^t_{\bF_p}[sl_n(\bF_p)^\vee], \Lambda^{t-1}_{\bF_p}[sl_n(\bF_p)^\vee]),
\end{align*}
such that $d_2 : H^s(\SL_n(\bZ/p) ; \Lambda^t_{\bF_p}[sl_n(\bF_p)^\vee]) \to H^{s+2}(\SL_n(\bZ/p) ; \Lambda^{t-1}_{\bF_p}[sl_n(\bF_p)^\vee])$ is given by Yoneda product with $d_2^t$. For each $t \geq 3$ odd, we therefore wish to evaluate the composition
$$\bF_p \overset{c_t}\lra \Lambda^t_{\bF_p}[sl_n(\bF_p)^\vee]  \overset{d_2^t}\lra  \Lambda^{t-1}_{\bF_p}[sl_n(\bF_p)^\vee][2]$$
as a morphism in the derived category of $\bF_p[\SL_n(\bZ/p)]$-modules, i.e.\ an element of $\Ext^2_{\bF_p[\SL_n(\bZ/p)]}(\bF_p, \Lambda^{t-1}_{\bF_p}[sl_n(\bF_p)^\vee]) = H^2(\SL_n(\bZ/p); \Lambda^{t-1}_{\bF_p}[sl_n(\bF_p)^\vee])$. The Leibniz rule as described in Lemma \ref{lem:d2} (iii), applied $t$-many times, gives a commutative diagram
\begin{equation*}
\begin{tikzcd}
(sl_n(\bF_p)^\vee)^{\otimes t} \dar \arrow[rrrrr, "\sum_{i=1}^t (-1)^{i-1} \mathrm{Id}^{\otimes i-1} \otimes d_2^1 \otimes \mathrm{Id}^{\otimes t-i}"] & & & & & \bigoplus_{i=1}^t (sl_n(\bF_p)^\vee)^{\otimes t-1}[2] \dar\\
\Lambda^t_{\bF_p}[sl_n(\bF_p)^\vee]  \arrow[rrrrr, "d_2^{t}"]& & & & &  \Lambda^{t-1}_{\bF_p}[sl_n(\bF_p)^\vee][2],
\end{tikzcd}
\end{equation*}
in the derived category of $\bF_p[\SL_n(\bZ/p)]$-modules. Writing $V_n = \bF_p^n$, defining
$$\tilde{d}_2^1 : V_n \otimes V_n^\vee \lra sl(V_n)^\vee \overset{d_2^1}\lra \bF_p[2],$$
and recalling the classes 
$$\tilde{c}_t \in H^0(\SL_n(\bF_p) ; \Lambda^t_{\bF_p}[V_n \otimes V_n^\vee]) = \Ext^0_{\bF_p[\SL_n(\bZ/p)]}(\bF_p, \Lambda^{t}_{\bF_p}[sl_n(\bF_p)^\vee])$$
defined in the proof of Corollary \ref{cor:CalcGps}, the diagram
\begin{equation*}
\begin{tikzcd}
 & (V_n \otimes V_n^\vee)^{\otimes t} \dar \arrow[rrrrr, "\sum_{i=1}^t (-1)^{i-1} \mathrm{Id}^{\otimes i-1} \otimes \tilde{d}_2^1 \otimes \mathrm{Id}^{\otimes t-i}"] & & & & & \bigoplus_{i=1}^t (V_n \otimes V_n^\vee)^{\otimes t-1}[2] \dar\\
\bF_p  \rar{c_t} \arrow[ru, "\tilde{c}_t"]& (sl_n(\bF_p)^\vee)^{\otimes t}  \arrow[rrrrr, "\sum_{i=1}^t (-1)^{i-1} \mathrm{Id}^{\otimes i-1} \otimes d_2^1 \otimes \mathrm{Id}^{\otimes t-i}"] & & & & & \bigoplus_{i=1}^t (sl_n(\bF_p)^\vee)^{\otimes t-1}[2],
\end{tikzcd}
\end{equation*}
commutes in the derived category of $\bF_p[\SL_n(\bZ/p)]$-modules: the triangle commutes by the definition of $c_t$ in terms of $\tilde{c}_t$, and the square commutes because it does so on each summand by definition of $\tilde{d}_2^1$.

\begin{lemma}\label{lem:WhatIsd2}
The element 
\begin{align*}
\tilde{d}_2^1 &\in \mathrm{Ext}^2_{\bF_p[\SL_n(\bZ/p)]}(V_n \otimes V_n^\vee, \bF_p)\\
& \quad\quad \cong H^2(\SL_n(\bZ/p); \mathrm{Hom}(V_n \otimes V_n^\vee, \bF_p))\\
& \quad\quad \cong H^2(\SL_n(\bZ/p); V_n \otimes V_n^\vee) \cong \bF_p\{x_1^{[1]}\} \otimes \bF_p\{\Sigma^{ad}_1\}
\end{align*}
corresponds to $x^{[1]}_1 \otimes (1)$.
\end{lemma}
\begin{proof}
There is a map of fibrations of pro-spaces
\begin{equation*}
\begin{tikzcd}
\{B\SL_n(\bZ/p^r, p)\}_r \dar \rar  & \{B\SL_n(\bZ/p^r)\}_r \dar \rar & B\SL_n(\bZ/p) \arrow[d, equals]\\
B\SL_n(\bZ/p^2, p) \rar & B\SL_n(\bZ/p^2) \rar & B\SL_n(\bZ/p).
\end{tikzcd}
\end{equation*}
It follows from the discussion in Section \ref{sec:AnalyticGps} that the map 
$$H^i(\SL_n(\bZ/p^2, p);\bF_p) \lra H^i_\mathrm{cts}(\SL_n(\bZ_p, p);\bF_p) = \colim_r H^i(\SL_n(\bZ/p^r, p);\bF_p)$$
on cohomology between the fibres is an isomorphism in degrees $i \leq 1$, and in degree 1 this cohomology is $sl_n(\bZ/p)^\vee$. Thus the element $d_2^1 \in \Ext^2_{\bF_p[\SL_n(\bZ/p)]}(sl_n(\bZ/p)^\vee, \bF_p)$ associated by Lemma \ref{lem:d2} to the top fibration sequence is equal to the corresponding element associated to the bottom sequence.

The bottom fibration sequence is described in Remark \ref{rem:Conventionx}. There is an isomorphism $I + p A \mapsto A \, \mathrm{mod}\, p : \SL_n(\bZ/p^2, p) \overset{\sim}\to sl_n(\bZ/p)$, and by Lemma \ref{lem:d2} (ii) the element 
$$d^1_2 \in \mathrm{Ext}^2_{\bF_p[\SL_n(\bZ/p)]}(sl_n(\bZ/p)^\vee, \bF_p) \cong H^2(\SL_n(\bZ/p) ; sl_n(\bZ/p))$$
associated to the bottom fibration sequence is \emph{minus} the class $e_n$ which classifies this abelian extension. In Remark \ref{rem:Conventionx} we defined the class $x = x^{[1]} \otimes (1) \in H^2(\SL_n(\bZ/p); V \otimes V^\vee)$ to be the image of the class $-e_n$ under the natural map $H^2(\SL_n(\bZ/p) ; sl_n(\bZ/p)) \to H^2(\SL_n(\bZ/p) ; V \otimes V^\vee)$. The latter map also sends $d_2^1$ to $\tilde{d}_2^1$ by definition, giving the required identity.
\end{proof}

\begin{lemma}
For $1 \leq i \leq t$ the composition $(\mathrm{Id}^{\otimes i-1} \otimes \tilde{d}_2^1 \otimes \mathrm{Id}^{\otimes t-i}) \circ \tilde{c}_t$ corresponds to 
$$x^{[1]}_{i} \otimes (1, 2, \ldots, t-1) \in  \bF_p\{x_j^{[1]} \, | \, j \in \ul{t-1}\} \otimes \bF_p\{\Sigma_{t-1}^{ad}\} \cong H^2(\SL(V) ; (V \otimes V^\vee)^{\otimes \ul{t-1}}),$$
where $x^{[1]}_t := x^{[1]}_{1}$.
\end{lemma}
\begin{proof}
We wish to evaluate the map
\begin{align*}
H^*(\SL(V);  (V \otimes V^\vee)^{\otimes \ul{t}}) &\otimes H^*(\SL(V) ; \mathrm{Hom}((V \otimes V^\vee)^{\otimes \ul{t}}, (V \otimes V^\vee)^{\otimes \ul{t-1}}))\\
& \lra H^*(\SL(V) ; (V \otimes V^\vee)^{\otimes \ul{t-1}})
\end{align*}
given by cup product followed by evaluation. To do so it is clearest to generalise, and describe the analogous composition map
\begin{equation}\label{eq:BigComposition}
\begin{aligned}
&H^*(\SL(V);  \mathrm{Hom}(V^{\otimes S} \otimes (V^\vee)^{\otimes T}, V^{\otimes A}
 \otimes (V^\vee)^{\otimes B})) \\
 &\quad\quad \otimes H^*(\SL(V) ; \mathrm{Hom}(V^{\otimes A} \otimes (V^\vee)^{\otimes B}, V^{\otimes X} \otimes (V^\vee)^{\otimes Y}))\\
&\quad\quad\quad\quad \lra H^*(\SL(V);  \mathrm{Hom}(V^{\otimes S} \otimes (V^\vee)^{\otimes T}, V^{\otimes X} \otimes (V^\vee)^{\otimes Y}))
\end{aligned}
\end{equation}
in a stable range. Adjoining over gives an isomorphism
$$H^*(\SL(V);  \mathrm{Hom}(V^{\otimes S} \otimes (V^\vee)^{\otimes T}, V^{\otimes A}
 \otimes (V^\vee)^{\otimes B})) \cong H^*(\SL(V);  V^{\otimes A \sqcup T} \otimes (V^\vee)^{\otimes B \sqcup S}),$$
and the latter is identified with $\Gamma_{\bF_p}[x]^{\otimes A \sqcup T} \otimes \bF_p\{\mathrm{Bij}(B \sqcup S, A \sqcup T)\}$ in a stable range by the discussion in Section \ref{sec:FullWBr}. As described in in Section \ref{sec:GraphicalInt}, elements here may be depicted as full walled Brauer diagrams from $(\emptyset,\emptyset)$ to $(A \sqcup T,B \sqcup S)$ where each strand is decorated by an $x^{[i]}$, as shown in Figure \ref{fig:End1New} (a) below. Equivalently, they may be depicted as full walled Brauer diagrams from $(S,T)$ to $(A,B)$ where each strand is decorated by an $x^{[i]}$, as shown in Figure \ref{fig:End1New} (b): the equivalence is given by considering $S$ as being above the wall on the left-hand side, and $T$ as being below the wall on the left-hand side.

\begin{figure}[h]
\begin{center}
\includegraphics{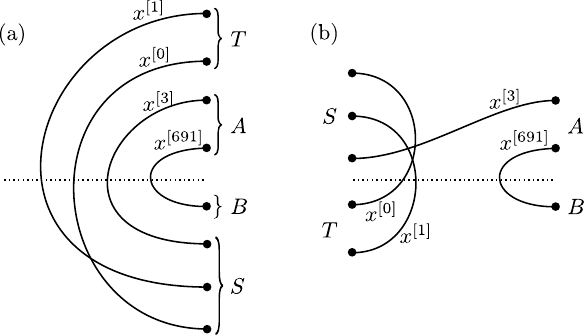}
\end{center}
\caption{(a) Elements of $\Gamma_{\bF_p}[x]^{\otimes A \sqcup T} \otimes \bF_p\{\mathrm{Bij}(B \sqcup S, A \sqcup T)\}$ as depicted in Section \ref{sec:GraphicalInt}. (b) An equivalent but more intuitive depiction; $S$ and $T$ have been moved to the other side of the wall, and rotated.}\label{fig:End1New}
\end{figure}

The latter style of depiction makes the composition map \eqref{eq:BigComposition} more intuitive. In this depiction it is given by the natural analogue of the discussion in Section \ref{sec:GraphicalInt}: concatenate such diagrams, multiply labels on the same strand together using the divided power multiplication on the $x^{[i]}$'s, then set closed components labelled by $x^{[i]}$ with $i>0$ to zero, and those labelled by $x^{[0]}$ to $\dim(V)$.

\begin{figure}[h]
\begin{center}
\includegraphics{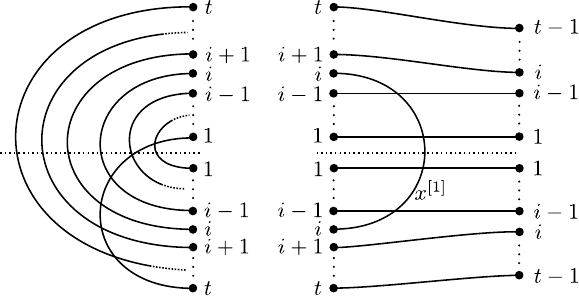}
\end{center}
\caption{The elements $\tilde{c}_t$ and $\mathrm{Id}^{\otimes i-1} \otimes \tilde{d}_2^1 \otimes \mathrm{Id}^{\otimes t-i}$. Unlabelled arcs are implicitly labelled $x^{[0]}$.}\label{fig:Big}
\end{figure}

Returning to the case at hand, the class $\tilde{c}_t \in H^0(\SL(V) ; (V \otimes V^\vee)^{\otimes \ul{t}})$ is represented by the left-hand diagram in Figure \ref{fig:Big}. By Lemma \ref{lem:WhatIsd2} the class $\mathrm{Id}^{\otimes i-1} \otimes \tilde{d}_2^1 \otimes \mathrm{Id}^{\otimes t-i} \in H^2(\SL(V) ; \mathrm{Hom}((V \otimes V^\vee)^{\otimes \ul{t}}, (V \otimes V^\vee)^{\otimes \ul{t-1}}))$ is represented by the right-hand diagram in Figure \ref{fig:Big}. Composing these as described above gives $x_i^{[1]} \otimes (1,2,\ldots, t-1)$, with the convention that $x^{[1]}_t := x^{[1]}_1$.
\end{proof}

In particular, for an odd $t \geq 3$ the class
$$\sum_{i=1}^t (-1)^{i-1} (\mathrm{Id}^{\otimes i-1} \otimes \tilde{d}_2^1 \otimes \mathrm{Id}^{\otimes t-i}) \circ \tilde{c}_t \in H^2(\GL(V) ; (V \otimes V^\vee)^{\otimes t-1})$$
mapped further to
$$H^2(\GL(V) ; \Lambda^{t-1}_{\bF_p}(V \otimes V^\vee)) = \left(\bF_p\{x_i^{[1]} \, | \, i \in \ul{t-1}\} \otimes \bF_p\{\Sigma_{t-1}^{ad}\}\right) \otimes_{\Sigma_{t-1}} \bF_p^-$$
is
$$\left[\sum_{i=1}^t (-1)^{i-1} x^{[1]}_{i} \otimes (1, 2, \ldots, t-1)\right] \otimes_{\Sigma_{t-1}} 1$$
(with $x^{[1]}_{t} := x^{[1]}_{1}$) and it follows, by acting on the $i$th term by the permutation $(1, 2, \ldots, t-1)^{-(i-1)} \in \Sigma_{t-1}$ of sign $(-1)^{i-1}$, that this is the same as
$$t \cdot \left[ x^{[1]}_{1} \otimes (1,2,\ldots, t-1)\right] \otimes_{\Sigma_{t-1}} 1.$$
When we map from $(V \otimes V^\vee)^{\otimes t-1}$ to $(sl(V)^\vee)^{\otimes t-1}$, according to the construction in Corollary \ref{cor:CalcGps} this is the element called $t \cdot e_{t-1}$.

This finishes the proof of Proposition \ref{prop:SSbehaviour}.

\begin{remark}[Stability range]\label{rem:StabRangeThmA}
Corollary \ref{cor:CalcGps} holds as long as $2 \cdot 2 \leq \dim(V)-2-2\cdot(p-1)$, i.e.\ $\dim(V) \geq 2p+4$, using that $\Lambda^i[sl(V)^\vee]$ is a summand of $(V_{[1,1]})^{\otimes i}$ for $i < p$, so is a coefficient system of degree $2i$, and the stability range in Section \ref{sec:stab}. Thus Proposition \ref{prop:SSbehaviour} holds for $n \geq 2p+4$, and so Theorem \ref{thm:ResIsZero} does too (in fact in the case $m>1$ it holds for all $n$). 

The lemma in Section \ref{sec:Ext} requires $H^1(\SL_n(\bZ/p) ; (V \otimes V^\vee)^{\otimes t})$ to be in the stable range for all $t < p-1$, which by the stability range in Section \ref{sec:stab} holds for $2 \cdot 1 \leq n -2 -2(p-2)$, i.e.\ $n \geq 2p$. 

The final ingredient in the proof of Theorem \ref{thm:mainCong} is homological stability for completed cohomology. In \cite{CEStability} an explicit range was not given, but the more recent work of Iwasa \cite[Theorem 1.3]{Iwasa} can be applied in our situation, and after taking (continuous) duals it shows that $\widetilde{H}^i(\SL_{n+1}(\bZ)) \to \widetilde{H}^i(\SL_{n}(\bZ))$ is an isomorphism for $2i \leq n-1$. In particular $\widetilde{H}^i(\SL_{n}(\bZ))$ is in the stable range for all $i<p$ as long as $n \geq 2p-1$. Thus Theorem \ref{thm:mainCong} holds as long as $n \geq 2p+4$.
\end{remark}

\section{Away from $p$}\label{sec:AwayFromChar} 

For completeness let us explain the analogue of Theorem \ref{thm:mainCong} with coefficients coprime to $p$. This is much simpler to analyse. It will be expressed in terms of the fibre of the map
$$\tilde{\kappa} : \SK(\bZ) \lra \SK(\bZ/p^m).$$

\begin{theorem}
Let $\ell$ be a prime different to $p$. In a stable range of homological degrees the group $\SL_n(\bZ/p^m)$ acts trivially on $H^*(\SL_n(\bZ, p^m);\bF_\ell)$, and the latter is isomorphic to $H^*(\Omega^\infty \mathrm{hofib}(\tilde{\kappa}) ; \bF_\ell)$ as an $\bF_\ell$-algebra.
\end{theorem}
\begin{proof}
It follows from work of Charney \cite{CharneyCong} that $H^*(\SL_n(\bZ, p^m) ; \bF_\ell)$ exhibits homological stability. Proposition 5.5  of that paper holds as condition (C1) asks that the inclusion $p^m\bZ \to \bZ$ induces an isomorphism on $\bF_\ell$-homology, which it does, and this provides the input for Theorem 5.2 of that paper. 

It then follows by an argument like that of Lemma \ref{lem:ActTrivOnSL} that $\SL_n(\bZ/p^m)$ acts trivially on $H^*(\SL_n(\bZ, p^m) ; \bF_\ell)$ in the stable range. In the map of fibre sequences
\begin{equation*}
\begin{tikzcd}[column sep=small]
B\SL_n(\bZ, p^m) \rar \dar & \Omega^\infty\mathrm{hofib}( \tilde{\kappa}) \dar\\
B\SL_n(\bZ) \dar \rar& \Omega^\infty \SK(\bZ) \dar{\Omega^\infty \tilde{\kappa}}\\
B\SL_n(\bZ/ p^m) \rar& \Omega^\infty \SK(\bZ/p^m) 
\end{tikzcd}
\end{equation*}
the bottom and middle maps are isomorphisms on homology in a stable range, so using the triviality of the $\SL_n(\bZ/p^m)$-action on $H^*(\SL_n(\bZ, p^m) ; \bF_\ell)$ it follows from the Zeeman comparison theorem that the top map is an isomorphism in a stable range too.
\end{proof}

\begin{corollary}
Let $\ell$ be a prime different to $p$. The map $\SL_n(\bZ, p^m) \to \SL_n(\bZ,p)$ induces an isomorphism on $\bF_\ell$-(co)homology in the stable range.
\end{corollary}
\begin{proof}
The kernel of $\SL_n(\bZ/p^m) \to \SL_n(\bZ/p)$ is a $p$-group so has trivial $\bZ[\tfrac{1}{p}]$-homology, so this map is a $\bZ[\tfrac{1}{p}]$-homology equivalence. Stabilising and plus-constructing, it follows that $\SK(\bZ/p^m)[\tfrac{1}{p}] \overset{\sim}\to \SK(\bZ/p)[\tfrac{1}{p}]$.
\end{proof}

Now Quillen \cite{QuillenFiniteFields} has calculated $H^*(\Omega^\infty \SK(\bZ/p) ; \bF_\ell)$ and $\pi_*(\SK(\bZ/p ; \bZ_\ell))$ completely, so one could perform an analysis similar to that of Section \ref{sec:CompletionMap} to describe $H^*(\SL_n(\bZ, p^m);\bF_\ell)$ in the stable range. We leave this to the interested reader.

\bibliographystyle{amsalpha}
\bibliography{biblio}

\end{document}